\documentclass[11pt,a4paper]{amsart}

\usepackage{a4wide}
\usepackage[utf8]{inputenc}
\usepackage{amsmath}
\usepackage{amsthm}
\usepackage{mathrsfs}
\usepackage{wasysym}
\usepackage{enumitem}
\usepackage{cite}
\usepackage[footnotesize]{caption}
\usepackage{color}   
\usepackage{wrapfig}
\usepackage{tgschola,eulervm}
\usepackage{comment}
\usepackage{amsfonts}
\usepackage{amssymb}
\usepackage{graphicx}
\usepackage[usenames,dvipsnames]{xcolor}

\usepackage{multicol}

\newcommand{\omitted}[1]{}

\usepackage[final]{hyperref}%

\newcommand{\heikodetail}[1]{}
\newcommand{\hide}[1]{}

\DeclareMathOperator{\opRot}{Rot}

\DeclareMathOperator{\dist}{dist}
\DeclareMathOperator{\lcm}{lcm}

\newcommand\ang{\mathop{\mbox{$<\!\!\!)$}}\nolimits}

\newcommand{\R}{\mathbb{R}}
\newcommand{\N}{\mathbb{N}}

\renewcommand{\S}{\mathbb{S}}
\newcommand{\Z}{\mathbb{Z}}

\newcommand\Id{{{\rm Id}}}

\newcommand{\TP}{\mathcal{T}}

\newcommand{\sB}{\mathscr{B}}
\newcommand{\sM}{\mathscr{M}}

\newcommand{\sL}{\mathscr{L}}

\newcommand{\DC}{\mathcal{K}}

\newcommand{\eps}{\varepsilon}

\newcommand{\g}{\gamma}

\definecolor{mygreen}{RGB}{30,50,0}

\newtheorem{thm}{Theorem}[section]
\newtheorem*{thm*}{Theorem}

\newtheorem{lem}[thm]{Lemma}
\newtheorem{cor}[thm]{Corollary}

\newtheorem*{introthm*}{Regularity Theorem}

\theoremstyle{definition}
\newtheorem{defin}[thm]{Definition}

\newtheorem{rem}[thm]{Remark}
\newtheorem*{rem*}{Remark}
\newtheorem{ex}[thm]{Example}

\renewcommand{\qed}{\hfill\mbox{$\Box$}\\}

\newcommand{\Fo}{\,\,\,\text{for }\,\,}
\newcommand{\Foa}{\,\,\,\text{for all }\,\,}

\newcommand{\AND}{\,\,\,\text{and }\,\,}

\def\li{\left(}
\def\ri{\right)}

\author{Alexandra Gilsbach}
\address[A.~Gilsbach]{
  \newline{}
  Institut f{\"u}r Mathematik
  \newline{}
  RWTH Aachen University
  \newline{}
  Templergraben~55
  \newline{}
  D-52062 Aachen, Germany}
\email{gilsbach@instmath.rwth-aachen.de}

\author{Heiko von~der~Mosel}
\address[H.~von~der~Mosel]{
  \newline{}
  Institut f{\"u}r Mathematik
  \newline{}
  RWTH Aachen University
  \newline{}
  Templergraben~55
  \newline{}
  D-52062 Aachen, Germany}
\email{heiko@instmath.rwth-aachen.de}


\date{\today}

\title[Critical torus knots
]
{Symmetric critical knots for O'Hara's energies
  }

\begin{document}
\frenchspacing

\begin{abstract}
We prove the existence of symmetric critical torus
knots for O'Hara's
knot energy family $E_\alpha$, $\alpha\in (2,3)$
using Palais' classic principle of symmetric
criticality. It turns out that in every torus knot class there are
at least two smooth $E_\alpha$-critical knots, which supports
experimental observations using numerical gradient flows.
\end{abstract}

\maketitle

\section{Introduction}   \label{sec:intro}
Experimenting with R. Scharein's computer program KnotPlot
\cite{knotplot_2017} 
L. H. Kauffman observed in  
\cite{kauffman_2012}  that there might be several distinct
local
minima present in the presumably complicated
knot energy landscape. In particular, 
a
numerical gradient flow implemented in KnotPlot may deform
different configurations of the same knot type into
distinct final states. For example, the observed shape of
the final knot configuration in the torus knot class
 $\TP(2,3)$ heavily depends on whether you start Scharein's flow
 with a $(2,3)$-- or with a $(3,2)$--representative; see 
 \cite[Section 3]{kauffman_2012}. Moreover, Kauffman reports
 the presence of a highly symmetrical $(3,4)$-torus knot as the final
 configuration of that flow 
 that does \emph{not} yield the absolute minimum of the
 energy.  We have made similar observations using Hermes' numerical
 gradient flow \cite{hermes_2014} for integral Menger curvature.  

 It is the purpose of this paper to support these experimental
 observations with rigorous analytic results establishing the
 existence of  \emph{at least two} symmetric critical knots in each 
 torus knot class.
Since Kauffman used Scharein's implementation of a Coulomb type
self-repulsion force according to an inverse power  of Euclidean
distance of different curve points, we focus here on
the family of self-repulsive potentials
\begin{equation}\label{ohara}
E_\alpha(\gamma):=\int_{\R/L\Z}\int_{-L/2}^{L/2}\left(
\frac{1}{|\gamma(u+w)-\g(u)|^\alpha}-\frac{1}{d_\g(u+w,u)^\alpha}\right)|\g'(u+w)||\g'(u)|\,dwdu
\end{equation}
for $\alpha\in [2,3)$, which forms a subfamily of J. 
O'Hara's energies
introduced in \cite{ohara_1992a}. Here,  $\g:\R/(L\Z)\to\R^3$,  
$L>0$, is a Lipschitz continuous closed curve, and
$$
d_\g(u+w,u):=\min\Big\{\sL\li\g|_{[u,u+w]}\ri,\sL(\g)-\sL\li\g|_{[u,u+w]}\ri\Big\} \Fo |w|\le L/2
$$
denotes the intrinsic distance, i.e., the length of the shorter arc on $\g$ connecting
the points $\g(u)$ with $\g(u+w)$. Here, the letter $\sL$ denotes the length of a curve.
\begin{rem}\label{rem:generalprop}
1.\,
For $\alpha=2$ the energy $E_2$ is called \emph{M\"obius energy} because of its invariance
under M\"obius transformations; see \cite[Theorem 2.1]{freedman-etal_1994}. For arbitrary
$\alpha\in [2,3)$ one still has invariance under isometries in $\R^3$ and under
reparametrizations.

2.\,
$E_2$ can be minimized in arbitrary prescribed 
prime knot classes according to Freedman, He, and Wang
\cite[Theorem 4.3]{freedman-etal_1994}, whereas
$E_\alpha$ for $\alpha\in (2,3)$ is minimizable in \emph{every} given tame knot class as shown
by O'Hara in \cite[Theorem 3.2]{ohara_1994}.

3.\,
For all $\alpha\in [2,3)$ the once-covered circle uniquely minimizes the energy $E_\alpha$, which was shown by Abrams et al. in \cite{abrams-etal_2003}.
\end{rem}

For the scaling-invariant version
\begin{equation}\label{scaledenergy}
S_\alpha:=\mathscr{L}^{\alpha-2}\cdot E_\alpha
\end{equation}
we prove the following central result.
\begin{thm}\label{thm:twocritical}
Let $a,b\in\Z\setminus\{0,\pm 1\}$ be relatively prime, $\alpha\in
(2,3)$. Then there are
at least  two arclength parametrized, embedded 
$S_\alpha$-critical
curves $\Gamma_1,\Gamma_2\in C^\infty(\R/\Z,\R^3)$
both representing the torus knot class $\TP(a,b)$,
such that there is no isometry
$I:\R^3\to\R^3$ with $I\circ \Gamma_1(\R/\Z)=\Gamma_2(\R/\Z).$
\end{thm}
In consequence, the gradient flow for $S_\alpha$ 
(or the flow for a linear combination of
$E_\alpha $ and length $\mathscr{L}$ treated
analytically by S. Blatt \cite{blatt_2018}) 
might
very well get stuck in one of these critical points
without having reached the absolute energy minimum.
Theorem \ref{thm:twocritical} could explain some of the experimental effects
described above --- in particular those displaying symmetric
non-minimizing final configurations since we use discrete
rotational symmetries
to construct $\Gamma_1$ and $\Gamma_2$.
However, Theorem \ref{thm:twocritical} contains
no statement about stability, so these $S_\alpha$-critical knots 
may be local minima or merely saddle points.

In contrast to the work of J. Cantarella et al. 
\cite{cantarella-etal_2014b} on symmetric criticality for the
non-smooth ropelength functional 
we obtain here
smooth critical points of the continuously differentiable
energy functional $S_\alpha$ since we can apply the classic
principle of symmetric criticality made rigorous by
R. Palais in \cite{palais_1979}. This principle can also
be applied to various types of geometric curvature
energies such as integral Menger curvature or tangent-point 
energies investigated in 
\cite{strzelecki-etal_2010,strzelecki-vdm_2012,strzelecki-etal_2013a}, to 
produce symmetric critical knots in any knot class that possesses at least
one symmetric representative.
Suitably scaled versions of those energies do converge to
ropelength in the $\Gamma$-limit sense as their
integrability exponents
tend to infinity. This implies,  in particular,
that the
symmetric critical knots we produce by Palais' principle
converge to symmetric
ropelength-critical knots;
see \cite{gilsbach_2018,gilsbach-knappmann-vdm_2017}. 
At this point, however, 
it is not clear if we thus obtain in the $\Gamma$-limit 
the same ropelength-critical points 
 as the ones 
Cantarella et al. provide in \cite{cantarella-etal_2014b}. 

The M\"obius energy, i.e., the case $\alpha=2$, is excluded
in  
Theorem \ref{thm:twocritical}; in
 ongoing work \cite{blatt-gilsbach-vdm_2017} we treat
this technically more challenging 
energy. D. Kim and R. Kusner, however, have chosen in
\cite{kim-kusner_1993} a different,
in a sense one-dimensional
approach to symmetric criticality for the M\"obius energy. They restrict
 their search to torus knots that actually
lie on the surfaces of 
tori foliating the $\S^3$ through variations of the  tori's radius ratio. 
It would be interesting to
investigate the relation between
their M\"obius-critical torus knots and the ones we aim for
in \cite{blatt-gilsbach-vdm_2017}.  
Kim and Kusner conjecture in \cite[p. 2]{kim-kusner_1993} on the
basis of their numerical experiments with Brakke's evolver
\cite{brakke_1992}
that stability of M\"obius critical torus knots in $\TP(a,b)$
should only be expected when $a=2$ or $b=2$. Stability  for
symmetric critical knots
is still an open problem not only for the scaled
O'Hara energies $S_\alpha$ 
but also 
for all other knot energies mentioned so far.

Let us briefly outline the structure of the paper. In Section
\ref{sec:psk} we recall the relevant aspects of Palais' principle
of symmetric criticality on Banach manifolds. The most 
important
properties of O'Hara's energies $E_\alpha$ are presented in Section
\ref{sec:ohara}, such as self-avoidance 
(Lemma \ref{lem:rough-bilip}), semicontinuity 
(Lemma \ref{lem:continuous}), and
Blatt's characterization \cite{blatt_2012a}
of energy spaces (Theorem \ref{thm:blatt})
in terms of fractional Sobolev spaces, 
so-called Sobolev-Slobodetckij  
spaces. This characterization is crucial in Section
\ref{sec:critical} to identify the correct Banach manifold (Corollary
\ref{cor:openset}), 
on which Palais' principle of symmetric criticality is applicable.
Then we describe
discrete rotational
symmetries of parametrized curves
in terms of a group action of the cyclic group (Definition
\ref{def:cyclic} and Lemma \ref{lem:groupaction}). After checking
the effects of reparametrizations on symmetry properties 
(Corollary \ref{cor:inherit-symmetry}) we focus on the torus knot
classes $\TP(a,b)$ to find symmetric representatives (Lemma
\ref{lem:symmetric-subset}), and use a direct method in the
calculus of variations to minimize $S_\alpha$ in symmetric
subsets (Theorem \ref{thm:existence}). Using well-known
knot theoretic periodicity properties of $\TP(a,b)$, we can 
finally identify two geometrically different symmetric
critical knots, which establishes Theorem \ref{thm:twocritical}.
This proof is based on a general result on possible rotational
symmetries for general tame
knots (Theorem \ref{thm:gruenbaum-gilsbach}), for 
which we present a purely geometric proof, and which may be of 
independent interest.
Some technical intermediate results, e.g. on the Sobolev-Slobodetckij
seminorm, or on sets invariant under discrete rotations, 
are proven in the appendix. 

The paper is essentially self-contained not only for the convenience
of the reader but also because in places we needed somewhat
more refined
versions of known results such as Theorem \ref{thm:blatt}.

\section{The principle of symmetric criticality}\label{sec:psk}
In this section we briefly recall the notion of a group action on a in general infinite dimensional
Banach manifold in order to formulate a version of Palais' principle of symmetric criticality
suitable for our application.
\begin{defin}\label{def:banach-mfd}
Let $k\in\N\cup\{0\}$ and $\mathscr{B}$ a Banach space. Then 
a Hausdorff space $\mathscr{M}$  is a \emph{Banach manifold modeled over $\mathscr{B}$
of class $C^k$}, or in short, a \emph{$C^k$-manifold over $\sB$}
if and only if the following two conditions hold:
\begin{enumerate}
\item[\rm (i)]
For all $x\in\sM$ there is an open set $V_x\subset\sM$ containing $x$,  and some
open set $\Omega_x\subset\sB$ containing $0$, and a homeomorphism $\phi_x:\Omega_x\to V_x$
with $\phi_x(0)=x.$
\item[\rm (ii)]
For two distinct points $x,y\in\sM$ with $x,y\in V_x\cap V_y$,  the corresponding
homeomorphisms $\phi_x:\Omega_x\to V_x\subset\sM$ and $\phi_y:\Omega_y\to V_y\subset\sM$ satisfy
$$
\phi_y^{-1}\circ\phi_x|_{\Omega_x\cap\Omega_y}\in C^k(\Omega_x\cap\Omega_y,\sB).
$$
\end{enumerate}
$\sM$ is a smooth, or $C^\infty$-manifold over $\sB$ if $\sM$ is a $C^k$-manifold over $\sB$ for
all $k\in\N$. The maps $\phi_x$ are called \emph{local parametrizations}, and their inverse mappings
$\phi_x^{-1}:V_x\to\Omega_x$ are the \emph{local charts}. The collection of all charts together with their respective domains forms a \emph{$C^k$-atlas} of the Banach manifold $\sM$.
\end{defin}

\begin{ex}
Every open subset $\Omega\subset\sB$ of a Banach space $\sB$ is a smooth manifold over $\sB$, since
for every $x\in\Omega$ one may choose the parametrization $\phi_x:=\Id_\sB$, so that the atlas of
this simple Banach manifold contains only one  element, namely $(\Id_\sB,\Omega)$.
\end{ex}

In order to incorporate symmetry in a mathematically rigorous way, one uses groups and their action
on Banach manifolds; cf. \cite[pp. 19,20]{palais_1979}.
\begin{defin}\label{def:groupact}
Let $(G,\circ)$ be a group,  $\sB$ a Banach space, and $\sM$ a $C^k$-manifold over $\sB$ for some
$k\in\N.$
\begin{enumerate}
\item[\rm (i)]
\emph{$G$ acts on $\sM$} if and only if there is a mapping $\tau:G\times\sM\to\sM$ mapping
a pair $(g,x)$ to a point $\tau_g(x)\in\sM$, such that
$$
\tau_{g\circ h}(x)=\tau_g(\tau_h(x))\Foa g,h\in G,\, x\in\sM.
$$
(Such a mapping $\tau$ is called a \emph{representation of $G$ in $\sM$}.)
\item[\rm (ii)]
$\sM$ is called a \emph{$G$-manifold (of class $C^k$)} if and only if for each $g\in G$
the mapping $\tau_g:\sM\to\sM$ is a $C^k$-diffeomorphism. If $G$ is an infinite Lie group then it is additionally required that the 
representation $\tau:G\times\mathscr{M}\to\mathscr{M}$ 
is of class $C^k$ for $\mathscr{M}$ to be a $G$-manifold. 
\item[\rm (iii)]
For a $G$-manifold the \emph{subset of $G$-symmetric points}, 
or in short the \emph{$G$-symmetric subset}
$\Sigma\subset\sM$ is defined as
$$
\Sigma:=\{x\in\sM:\tau_g(x)=x\Foa g\in G\}.
$$
\item[\rm (iv)]
A function $E:\sM\to\R$ is \emph{$G$-invariant} if and only if
$$
E(\tau_g(x))=E(x)\Foa g\in G,\,x\in \sM.
$$

\end{enumerate}
\end{defin}
Now, Palais' principle of symmetric criticality reads as follows; 
cf. \cite[Thm.5.4]{palais_1979}.
\begin{thm}[Palais]\label{thm:palais}
Let $G$ be a compact Lie group and $\sM$ a $G$-manifold of class $C^1$ over the Banach space 
$\sB$ with $G$-symmetric subset $\Sigma\subset\sM$, and let $E:\sM\to\R$ be a
$G$-invariant function of class $C^1$. Then $\Sigma $ is a $C^1$-submanifold of
$\sM$, and $x\in\Sigma$ is a critical point of $E$ if and only if $x$ is
critical for $E|_\Sigma:\Sigma\to\R.$
\end{thm}
Since any finite group is a Lie group \cite[p. 48, Example 5]{cohn_1957} one immediately obtains the following result which will be of relevance in our application.
\begin{cor}\label{cor:palais}
If $G$ is a finite group, $\sM$ a $G$-manifold of class $C^1$ over the Banach space $\sB$ with
$G$-symmetric subset $\Sigma\subset\sM$, and if $E:\sM\to\R$ is a $G$-invariant function of class
$C^1$, then $x\in\Sigma$ is $E$-critical if and only if it is $E|_\Sigma$-critical.
\end{cor}
\begin{rem}
In our application the Banach manifold $\sM$ will be an open subset $\Omega\subset\sB$
of a
Banach space $\sB$, so that the differential of a $C^1$-function $E:\Omega\to\R$ coincides
with the classic Fr\'echet-differential 
$$
dE_x:T_x\Omega\simeq\sB\to T_{E(x)}\R\simeq\R,
$$
which may be calculated using the first variation, or G\^{a}teaux-derivative:
\begin{equation}\label{gateaux}
dE_x[h]=\delta E(x,h):=\lim_{\eps\to 0}\frac{E(x+\eps h)-E(x)}{\eps}\Fo h\in\sB.
\end{equation}
Theorem \ref{thm:palais} then implies that in order to establish criticality of a point $x\in\Sigma$
it suffices 
to show 
$$
dE_x[h]=0\Foa h\in T_x\Sigma,
$$
and not for all $h\in\sB$.
\end{rem}

\section{Properties of O'Hara's knot energies $E_\alpha$}\label{sec:ohara}
We start with 
 the following bi-Lipschitz estimate
due to O'Hara \cite[Theorem 2.3]{ohara_1992a}, whose proof we present here for the convenience of
the reader.
\begin{lem}\label{lem:rough-bilip}
Any $\g\in C^{0,1}(\R/\Z,\R^3)$  
with $|\gamma'|>0$ a.e. and with $E_\alpha(\g)<\infty$ for some $\alpha\in [2,3)$ is injective. 
More precisely, for all $b\ge 0$ there is a constant $C=C(b)\ge 0$ such that
$E_\alpha(\gamma)\le b$ implies the bi-Lipschitz estimate
\begin{equation}\label{rough-bilip}
|\gamma(s)-\gamma(t)|\ge Cd_\gamma(s,t)\Foa s,t\in \R/\Z.
\end{equation}
\end{lem}
\begin{proof}
Since $|\gamma'|>0$ a.e. there is a one-to-one correspondence between
the original parameters $s,t\in\R/\Z$ and the respective arclength 
parameters $\sigma(s)=\int_0^s|\gamma'(\tau)|\,d\tau$ and
$\sigma(t)=\int_0^t|\gamma'(\tau)|\,d\tau$, so we may assume without loss of generality that $\gamma$ is already parametrized according
to arclength, i.e.,
$|\gamma'(\tau)|=1$ for a.e. $\tau\in \R/\Z$, and (by a parameter shift)
that 
$$
0\le s<t\le s+\frac 12.
$$
Consequently,  $(t-s)=|s-t|$ which equals the intrinsic distance  
$$
d_\gamma(s,t)=|s-t|_{\R/\Z}:=\min\{|s-t|,1-|s-t|\}.
$$
Setting
$$
d:=|\gamma(s)-\gamma(t)|\AND\delta:=(t-s)
$$ 
we assume first that $d\le\delta/4,$ so that we can estimate for $0\le u,v\le\delta/8$
\begin{equation}\label{ohara1}
|\gamma(s+u)-\gamma(t-v)|\le d+u+v
\end{equation}
and
\begin{equation}\label{ohara2}
|(t-v)-(s+u)|=(t-s)-(u+v)=\delta-(u+v)\ge \frac 34 \delta,
\end{equation}
where, again,  the left-hand side equals the intrinsic distance $d_\gamma(s+u,t-v)$. 
By means of \eqref{ohara1} and \eqref{ohara2} we may now bound the energy from below to obtain
\begin{eqnarray}
b & \ge & \int_s^{s+\tfrac{\delta}{8}}\int_0^{t-\tfrac{\delta}{8}}\left(\frac{1}{|\gamma(x)-
\gamma(y)|^\alpha}-\frac{1}{d_\gamma(x,y)^\alpha}\right)\,dydx\notag\\
& =  &\int_0^{\tfrac{\delta}{8}}\int_0^{\tfrac{\delta}{8}}\left(\frac{1}{|\gamma(s+u)-\gamma(t-v)|^\alpha}-
\frac{1}{|(t-v)-(s+u)|^\alpha}\right)\,dudv\notag\\
&\overset{\eqref{ohara1},\eqref{ohara2}}{\ge} &
\int_0^{\tfrac{\delta}{8}}\int_0^{\tfrac{\delta}{8}}\left(\frac{1}{(d+u+v)^\alpha}-\frac{1}{(\tfrac 34 \delta)^\alpha}\right)\,dudv\notag\\
& = & \int_0^{\tfrac{\delta}{8}}\int_0^{\tfrac{\delta}{8}}\frac{1}{(d+u+v)^\alpha}\left[1-
\left(\frac{d+u+v}{\tfrac 34 \delta}\right)^\alpha\right]\,duv.\label{b-absch}
\end{eqnarray}
To estimate the term in square brackets in \eqref{b-absch} notice that $d+u+v\le d+(\delta/4)\le
\delta/2$ so that $(d+u+v)/(3\delta/4)\le 2/3$, from which we infer
$$
b  \ge  \Big(1-\big(\tfrac 23 \big)^\alpha\Big) \int_0^{\tfrac{\delta}{8}}\int_0^{\tfrac{\delta}{8}}\frac{1}{(d+u+v)^2}\,dudv
  =  \Big(1-\big(\tfrac 23 \big)^\alpha\Big)
  \log\frac{(d+\tfrac{\delta}8)^2}{d(d+\tfrac{\delta}4)}
$$
by explicit integration. With $d+(\delta/8)\ge (d+(\delta/4))/2$ we can bound the argument of the logarithm by $\delta/(16d)$ from below to obtain
$$
b\ge\Big(1-\big(\tfrac 23 \big)^\alpha\Big) \log\frac{\delta}{16d},
$$ 
which leads to $e^{b/(1-(2/3)^\alpha)}\ge \delta/(16d)$ or
\begin{equation}\label{b-final}
d\ge \frac{1}{16}e^{-b/(1-(2/3)^\alpha)}\delta \quad\textnormal{if $d\le \delta/4.$}
\end{equation}
This verifies our claim with constant 
$$
C:=\min\Big\{\frac 14, \frac{1}{16}e^{-b/(1-(2/3)^\alpha)}\Big\}=\frac{1}{16}e^{-b/(1-(2/3)^\alpha)}.
$$
\end{proof}
Crucial for the application of Palais' principle of symmetric criticality is the identification
of a suitable Banach manifold   in our context of knotted curves and O'Hara's energy $E_\alpha$. This will be an open subset of 
an appropriate Sobolev-Slobodetckij space, which -- according
to the important contribution of Blatt \cite{blatt_2012a} -- characterizes curves of finite 
$E_\alpha$-energy. Here is a slightly refined statement of Blatt's theorem.
\begin{thm}[Blatt]\label{thm:blatt}
For any $\alpha\in [2,3)$ the following is true.
\begin{enumerate}
\item[\rm (i)]
If $\g\in C^{0,1}(\R/\Z,\R^3)$ with length $0<L:=\sL(\g)$ satisfies $|\g'|>0$ a.e. and
$E_\alpha(\g)<\infty$, then $\g|_{[0,1)}$ is injective, and its arclength parametrization
$\Gamma\in C^{0,1}(\R/(L\Z),\R^3) $ is of class $W^{(\alpha +1)/2,2}(\R/(L\Z),\R^3)$
with unit  tangent $\Gamma'$ satisfying
\begin{equation}\label{energy-estimate1}
[\Gamma']^2_{(\alpha-1)/2,2}\le 4^4\cdot 2^{2-2\alpha}E_\alpha(\gamma).
\end{equation}
\item[\rm (ii)]
If, on the other hand, $\alpha\in (2,3)$ and
$\g\in W^{(\alpha +1)/2,2}(\R/\Z,\R^3)$ with $|\g'|>0$ a.e., and if
$\g|_{[0,1)}$ is injective, then $E_\alpha(\g)<\infty.$
\end{enumerate}
\end{thm}
Blatt actually proved part (ii) only for arclength parametrized curves, but for the full two-parameter family of O'Hara's energies which also includes the case $\alpha=2$.

Before giving the proof of Theorem \ref{thm:blatt}
let us quickly recall the concept of Sobolev-Slobodetckij spaces, where 
it suffices for our applications to focus  on the case of periodic functions of one variable.
For that we define for fixed $L>0$, $s\in (0,1)$ and $\rho\in [1,\infty)$ the seminorm
\begin{equation}\label{seminorm}
[f]_{s,\rho} :=\left(\int_{\R/(L\Z)}\int_{-L/2}^{L/2}
\frac{|f(u+w)-f(u)|^\rho}{|w|^{1+\rho s}}\,dwdu\right)^{1/\rho}
\end{equation}
for an integrable function $f\in L^\rho(\R/(L\Z),\R^n)$, which explains our notation
in \eqref{energy-estimate1}.
\begin{defin}\label{def:sobolev-slobo}
For $k\in\N$, the set
$$
W^{k+s,\rho}(\R/(L\Z),\R^n):=\{f\in W^{k,\rho}(\R/(L\Z),\R^n):\|f\|_{W^{k+s,\rho}}<\infty\},
$$
where 
$$
\|f\|_{W^{k+s,\rho}}:=\|f\|_{W^{k,\rho}}+[f^{(k)}]_{s,\rho },
$$
is called the \emph{Sobolev-Slobodetckij space} with (fractional) differentiability order
$k+s$ and integrability $\rho.$ 
(Here, $W^{k,\rho}$ denotes the usual Sobolev space of functions whose generalized
derivatives  up to order $k$ are $\rho$-integrable.)
\end{defin}
\begin{rem}\label{rem:sobolev-slobo}
It is well-known that Sobolev-Slobodetckij are Banach spaces, and one has the following 
continuous Morrey-type embedding\footnote{For this and many more advanced facts on fractional
Sobolev spaces we refer, e.g., to \cite{runst-sickel_1996}, \cite{valdinoci-etal_2012}, 
\cite{bergh-loefstroem_1976}, or to the monographs \cite{triebel_1983, triebel_1992, triebel_2006}} into classical H\"older spaces: 
$$
W^{k+s,\rho}(\R/(L\Z),\R^n)\hookrightarrow C^{k,s-(1/\rho)}(\R/(L\Z),\R^n)\Fo \rho\in (1,\infty),
s\in (1/\rho,1).
$$
In our context we obtain for $\alpha\in (2,3),$ $s=(\alpha-1)/2\in (1/2,1)$, and $\rho=2$ 
the continuous embedding
\begin{align}\label{embedding}
W^{(\alpha+1)/2,2}(\R/(L\Z),\R^n)=W^{1+s,2}(\R/(L\Z),\R^n)&\hookrightarrow C^{1,s-(1/2)}(\R/(L\Z),\R^n)
\notag\\
&
=  C^{1,(\alpha/2)-1}(\R/(L\Z),\R^n),
\end{align}
which means that there is a constant $C_E=C_E(L,n)$  
such that 
\begin{equation}\label{embeddingineq}
\|f\|_{C^{1,(\alpha/2)-1}}\le C_E\|f\|_{W^{(\alpha+1)/2,2}}\Foa f\in W^{(\alpha+1)/2,2}(\R/(L\Z),\R^n).
\end{equation}
This uniform estimate will turn out to be quite useful in our context, e.g., to obtain compactness,
or to 
conserve the prescribed knot class in the limit of minimal sequences; see Section \ref{sec:critical}, 
in particular the proof of Theorem \ref{thm:existence}.
\end{rem}

{\it Proof of Theorem \ref{thm:blatt}.}\,
(i)\, Injectivity follows from Lemma \ref{lem:rough-bilip}. 
Since $E_\alpha$ is invariant under reparametrization we have $E_\alpha(\g)=E_\alpha(\Gamma)$.
So, we can estimate
\begin{eqnarray}
\infty>E_\alpha(\Gamma) & = & \int_{\R/(L\Z)}\int_{-L/2}^{L/2}\left(\frac{1-\frac{|\Gamma(u+w)-
\Gamma(u)|^\alpha}{|w|^\alpha}}{|w|^\alpha}\right)\frac{|w|^\alpha}{|\Gamma(u+w)-
\Gamma(u)|^\alpha}\,dwdu\notag\\
& \ge & \int_{\R/(L\Z)}\int_{-L/2}^{L/2}\frac{1-\frac{|\Gamma(u+w)-
\Gamma(u)|^2}{|w|^2}}{|w|^\alpha}\,dwdu,\label{zwischen1}
\end{eqnarray}
where we have used that the arclength parametrization
$\Gamma$ is Lipschitz continuous with Lipschitz constant $1$, and $\alpha \ge 2.$
Now, the numerator of the last integral may be rewritten as
\begin{equation}\label{stepa}
\int_0^1\int_0^1\Big(1-\Gamma'(u+\sigma w)\cdot\Gamma'(u+\tau w)\Big)\,d\sigma d\tau=\frac 12
\int_0^1\int_0^1|\Gamma'(u+\sigma w)-\Gamma'(u+\tau w)|^2\,d\sigma d\tau, 
\end{equation}
which -- inserted into \eqref{zwischen1} and combined with Fubini's theorem -- leads to the 
following lower bound for $E_\alpha(\Gamma)$:
\begin{equation}\label{stepb}
\frac 12 \int_0^1\int_0^1\int_{-L/2}^{L/2}\int_0^L\frac{|\Gamma'(u+\sigma w)-\Gamma'(u+\tau w)|^2}{
|w|^\alpha}\,dudwd\sigma d\tau,
\end{equation}
which can be transformed via the substitution $z:=u+\sigma w$ into
\begin{equation}\label{stepc}
\frac 12 \int_0^1\int_0^1\int_{-L/2}^{L/2}\int_{\sigma w}^{L+\sigma w}\frac{|\Gamma'(z)-\Gamma'(
z+(\tau-\sigma)w)|^2}{
|w|^\alpha}\,dzdwd\sigma d\tau.
\end{equation}
By $L$-periodicity we may replace the inner integration by the integral on $\R/(L\Z)$, 
and we estimate the resulting quadruple integral from below by restricting the integration with
respect to $\tau$ to the interval $[3/4,1]$ and the $\sigma$-integration to $[0,1/4]$,
before we interchange the inner two integrations with
 Fubini and substitute then $y:=(\tau-\sigma)w$, to arrive at the new lower bound for 
 $E_\alpha(\Gamma)$:
$$
\frac 12 \int_{3/4}^1\int_0^{1/4}(\tau-\sigma)^{\alpha-1}
\int_{\R/(L\Z)}\int_{-(\tau-\sigma)L/2}^{(\tau-\sigma)L/2}\frac{|\Gamma'(z)-\Gamma'(z+y)|^2}{|y|^\alpha}\,dydzd\sigma d\tau,
$$
which itself is bounded from below by
\begin{equation}\label{esti1}
\frac 12 \Big( \frac 14 \Big)^3\int_{\R/(L\Z)}\int_{-L/4}^{L/4}
\frac{|\Gamma'(z)-\Gamma'(z+y)|^2}{|y|^\alpha}\,
dydz.
\end{equation}
The Sobolev-Slobodetckij seminorm \eqref{seminorm} for $s=(\alpha-1)/2$ and therefore $1+2s=
\alpha$, on the other hand, 
may be estimated by means of the triangle inequality
as
\begin{eqnarray}
[\Gamma']^2_{(\alpha-1)/2,2} & = & \int_{\R/(L\Z)}\int_{-L/2}^{L/2}
\frac{|\Gamma'(z+x)-\Gamma'(z)|^2}{|x|^\alpha}\,dxdz\notag\\
& \le &  2\int_{\R/(L\Z)}\int_{-L/2}^{L/2}
\frac{|\Gamma'(z+x)-\Gamma'(z+(x/2))|^2}{|x|^\alpha}\,dxdz\notag\\ 
& & +
2\int_{\R/(L\Z)}\int_{-L/2}^{L/2}
\frac{|\Gamma'(z+(x/2))-\Gamma'(z)|^2}{|x|^\alpha}\,dxdz.\label{ingred0}
\end{eqnarray}
Substituting $y:=x/2$ transforms the second double integral on the right-hand side 
into
\begin{equation}\label{ingred1}
2^{1-\alpha}\int_{\R/(L\Z)}\int_{-L/4}^{L/4}\frac{|\Gamma'(z+y)-\Gamma'(z)|^2}{
|y|^\alpha}\,dydz.
\end{equation}
In the first integral on the right-hand side of \eqref{ingred0} we first use Fubini
to interchange the order of integration, then the substitution $\zeta:=z+x$ in the
$z$-integral to arrive at
$$
\int_{-L/2}^{L/2}\int_{x}^{L+x}\frac{|\Gamma'(\zeta)-\Gamma'(\zeta-(x/2))|^2}{
|x|^\alpha}\,d\zeta dx=\int_{-L/2}^{L/2}\int_{\R/(L\Z)}\frac{|\Gamma'(\zeta)-\Gamma'(\zeta-(x/2))|^2}{
|x|^\alpha}\,d\zeta dx,
$$
where we used $L$-periodicity of $\Gamma'$.  Interchanging the order of integration again, 
and then substituting here $y:=-x/2$ in the $x$-integration   finally leads to
the term \eqref{ingred1} again. Thus, inserting \eqref{ingred1} for both double 
integrals on the right-hand side of \eqref{ingred0}, and combining this with
\eqref{esti1} we obtain the desired energy estimate \eqref{energy-estimate1}.

(ii)\,
By Lemma \ref{lem:arclength-seminorm} also the arclength parametrization $\Gamma:\R/(L\Z)\to\R^3$ of $\g$
is of class $W^{(\alpha+1)/2,2} $ with the estimate \eqref{seminormest-arclength},
where $L=\sL(\g)$ denotes the length of $\g$.
So, it suffices to work with $\Gamma$ due to the parameter invariance of $E_\alpha$.
In addition, we prove in the appendix (see Corollary \ref{cor:arclength-bilip}) 
that ${\Gamma}$ is bi-Lipschitz continuous
satisfying
\begin{equation}\label{arclength-bilip}
\frac{1}{{B}}|w|\le |{\Gamma}(u+w)-{\Gamma}(u)|\le |w|\Foa u\in\R/(L\Z),\, |w|\le L/2
\end{equation}
for some constant ${B}={B}(\alpha,{\Gamma})$ 
depending on $\alpha $ and on the curve 
${\Gamma}$.
Similarly
as in the proof of part (i) we first rewrite the energy of ${\Gamma}$
as
\begin{eqnarray}
E_\alpha({\Gamma}) & = & \int_{\R/(L\Z)}\int_{-L/2}^{L/2}\left(
\frac{1-\frac{|{\Gamma}(u+w)-
{\Gamma}(u)|^\alpha}{|w|^\alpha}}{|w|^\alpha}\right)\frac{|w|^\alpha}{|{\Gamma}(u+w)-
{\Gamma}(u)|^\alpha}\,dwdu\notag\\
& {\le} & 
{B}^\alpha\int_{\R/(L\Z)}\int_{-L/2}^{L/2}\frac{1-\frac{|{\Gamma}(u+w)-
{\Gamma}(u)|^\alpha}{|w|^\alpha}}{|w|^\alpha}\,dwdu,\label{zwischen4}
\end{eqnarray}
where we used \eqref{arclength-bilip} for the inequality. By the elementary inequality
$$
1-x^\alpha\le (\alpha+1)(1-x^2) \Foa \alpha\in [2,\infty),\, x\in [0,1]
$$
proved in Lemma \ref{lem:elem-ineq} in the appendix we can estimate the right-hand side of 
\eqref{zwischen4} from above by 
\begin{equation}\label{zwischen5}
(\alpha+1){B}^\alpha\int_{\R/(L\Z)}\int_{-L/2}^{L/2}\frac{1-\frac{|{\Gamma}(u+w)-
{\Gamma}(u)|^2}{|w|^2}}{|w|^\alpha}\,dwdu.
\end{equation}
This double integral is identical with the one in \eqref{zwischen1}, 
so we can perform
exactly the same manipulations using Fubini and one  substitution  as in
\eqref{stepa},\eqref{stepb}, \eqref{stepc}, to rewrite
\eqref{zwischen5}
as
\begin{equation}\label{zwischen6}
\frac 12 (\alpha+1){B}^\alpha\int_0^1\int_0^1\int_{-L/2}^{L/2}\int_{\sigma w}^{L+\sigma w}
\frac{|{\Gamma}'(z)-{\Gamma}'(
z+(\tau-\sigma)w)|^2}{
|w|^\alpha}\,dzdwd\sigma d\tau,
\end{equation}
where in the $z$-integration we may replace the domain of integration
by $\R/(L\Z)$ due to
$L$-periodicity of ${\Gamma}$. Exchanging the order of the $z$-integration with
the $w$-integration we can substitute $y(w):=(\tau-\sigma)w$ to obtain
\begin{equation}\label{zwischen7}
\frac 12 (\alpha+1){B}^\alpha\int_0^1\int_0^1
|\tau-\sigma|^{\alpha-1}
\int_{\R/(L\Z)}
\int_{-|\tau-\sigma|L/2}^{|\tau-\sigma|L/2}\frac{|{\Gamma}'(z)-{\Gamma}'(
z+y)|^2}{|y|^\alpha}\,dydzd\sigma d\tau,
\end{equation}
where the integration domain of the $y$-integration may be replaced by the full
interval $[-L/2,L/2]$ since $|\tau-\sigma|\le 1$, giving 
$$
\frac 12 (\alpha+1){B}^\alpha\int_0^1\int_0^1
\int_{\R/(L\Z)}
\int_{-L/2}^{L/2}\frac{|{\Gamma}'(z)-{\Gamma}'(
z+y)|^2}{
|y|^\alpha}\,dydzd\sigma d\tau=\frac 12 (\alpha+1){B}^\alpha[\Gamma']^2_{(\alpha-1)/2,2}
$$
as an upper bound for $E_\alpha(\Gamma)$. Combining this with
\eqref{seminormest-arclength} in Lemma \ref{lem:arclength-seminorm} in the appendix
we conclude
\begin{align}
E_\alpha(\gamma)=E_\alpha(\Gamma)&\le \frac 12 (\alpha+1){B}^\alpha[\Gamma']^2_{{(\alpha-1)/2,2}}
\notag\\
& \overset{\eqref{seminormest-arclength}}{\le}
\frac 12 (\alpha+1){B}^\alpha
\Big( \frac 1c \Big)^{2+\alpha}\Big[\Big(\frac 1c \Big)^2+C^{6}
 \Big]\cdot [\gamma']^2_{{(\alpha-1)/2,2}},\label{final-energy-est}
 \end{align}
 where $c=\min_{[0,1]}|\g'|$ and $C=\max_{[0,1]}|\g'|$,
 which finishes the proof.
\qed
Lower semicontinuity
of $E_\alpha$ was shown in the case $\alpha=2$ by Freedman, He, and Wang
in \cite[Lemma 4.2]{freedman-etal_1994}, and their argument works also for any $\alpha\in
[2,3)$. 
\begin{lem}\label{lem:continuous}
Let $\alpha\in [2,3)$ and assume that $\g,\g_i:\R/\Z\to\R^n$ are absolutely continuous
curves with $|\g'|>0$ and $|\g_i'|>0$ a.e. on $\R/\Z$ for all $i\in\N$, such that
$\g_i\to \g$ pointwise everywhere on $\R/\Z$ as $i\to\infty$.   Then 
$$
E_\alpha(\g)\le\liminf_{i\to\infty}E_\alpha(\g_i).
$$

\end{lem}
\begin{proof}
We may assume that the $\liminf$ on the right-hand side is finite, 
and that it is realized as the limit
$E_\alpha(\g_i)$ (upon restriction to a subsequence
again denoted by $\g_i$).
It is well-known that the length functional $\sL$ is lower semicontinuous with respect
to pointwise convergence, so that also
$d_\g(u+w,u)\le\liminf_{i\to\infty}d_{\g_i}(u+w,u)$; hence
$$
\limsup_{i\to\infty}\frac{1}{d_{\g_i}(u+w,u)}\le\frac{1}{d_\g(u+w,u)}\Foa u\in\R/\Z,\,|w|\le 1/2.
$$
Together with the pointwise convergence
$|\g_i(u+w)-\g_i(u)|\to 
|\g(u+w)-\g(u)|$ as $i\to\infty$ we obtain 
\begin{multline}
\frac{1}{|\g(u+w)-\g(u)|^\alpha}-
\frac{1}{d_\g(u+w,u)^\alpha}\le\\
\liminf_{i\to\infty}\left(\frac{1}{|\g_i(u+w)-\g_i(u)|^\alpha}-
\frac{1}{d_{\g_i}(u+w,u)^\alpha}\right)\Foa u\in\R/\Z,\,|w|\le 1/2.\label{integrand1}
\end{multline}
In addition, using again the lower semicontinuity of length, we can estimate for any
$0<h\ll 1$ and any $s\in\R/\Z$,
$$
|\g(s+h)-\g(s)| \le  d_\g(s+h,s)\le\liminf_{i\to\infty}d_{\g_i}(s+h,s)=
\liminf_{i\to\infty} \int_s^{s+h}|\g_i'(\tau)|\,d\tau.
$$
Dividing this inequality by $h$ and taking the limit $h\searrow 0$ we obtain at differentiability
points $s$ of $\gamma$
that are also Lebesgue points of   all $|\g_i'|$ simultaneously 
-- hence for
a.e. $s\in  \R/\Z $ -- the limiting inequality 
\begin{equation}\label{integrand2}
|\g'(s)|\le\liminf_{i\to\infty}|\g_i'(s)|.
\end{equation}
Combining \eqref{integrand1} with \eqref{integrand2} we obtain that
the integrand of $E_\alpha$
is bounded from above by the limes inferior of the integrands of $E_\alpha(\g_i)$ as $i\to\infty$.
This together with Fatou's Lemma and the monotonicity of the integral proves the claim.
\end{proof}

\begin{rem}\label{rem:C1}
In \cite[Theorem 1.1]{blatt-reiter_2013} Blatt and Reiter
prove that $E_\alpha$ is continuously differentiable on the space of
all injective regular curves of class $W^{(\alpha+1)/2,2}$, and they
give an explicit formula of the differential $dE_\g[\cdot]$ in the case of
an arclength parametrized curve $\g\in W^{(\alpha+1)/2,2}(\R/\Z,\R^n)$.
The explicit structure of this differential is not needed in our context,
but the differentiability of $E$ is, of course, crucial to apply Palais' 
principle of symmetric criticality to obtain classic critical points -- in contrast,
e.g. to the notion of criticality for the non-smooth ropelength functional formulated
by Cantarella et al. in \cite{cantarella-etal_2014b}.  Moreover,
Blatt and Reiter's main theorem \cite[Theorem 1.2]{blatt-reiter_2013}
states that any arclength parametrized critical point  of the linear combination
$E_\alpha +\lambda\sL$  is $C^\infty$-smooth. Here, $\sL$ denotes as before the
length functional, and $\lambda\in\R$ is an arbitrary parameter, that, e.g., comes
up as a Lagrange parameter for a minimization problem for $E_\alpha$ under a
fixed length constraint. 
Alternatively, and important for our construction of symmetric critical points in 
Section \ref{sec:critical}, such a scalar parameter appears if one
considers the scale-invariant version $S_\alpha$ of $E_\alpha$ defined in \eqref{scaledenergy} in the introduction. 
The differential of $S_\alpha$ 
evaluated at some injective regular curve 
$\g\in W^{(\alpha+1)/2,2}(\R/\Z,\R^n)$
has the form 
$$
d(S_\alpha)_\gamma=d\big(\sL^{\alpha-2}E_\alpha)_\gamma=\sL(\g)^{\alpha-2}d\big(E_\alpha\big)_\g+
\big((\alpha-2)\sL(\g)^{\alpha-1}E_\alpha(\g)\big)d\sL_\gamma.
$$
Hence Blatt and Reiter's regularity theorem applies to any arclength parametrized
critical point $\g$ of $S_\alpha$
(setting $\lambda:=(\alpha-2)\sL(\gamma) E_\alpha(\g)$)  
implying the smoothness of such $\g$. 
\end{rem}

\section{Critical torus knots}\label{sec:critical}
We first establish an open subset of the Banach space $W^{(\alpha+1)/2,2}(\R/\Z,\R^3)$ as
the Banach manifold on which Palais' principle of symmetric criticality is applicable.
\begin{lem}\label{lem:openset}
For any tame\footnote{A
 knot class is called \emph{tame}
  if it contains polygonal loops.
   Any knot class containing $C^1$-representatives is tame, see
    R. H. Crowell and R. H. Fox~\cite[App.~I]{crowell-fox_1977}, and vice versa, any tame knot class
     contains smooth representatives.}
knot class $\DC$ and for any $\alpha\in (2,3)$ the set
$$
\Omega_\DC:=\{\g=(\g^1,\g^2,\g^3)\in W^{(\alpha+1)/2,2}(\R/\Z,\R^3):
|\g'|>0, (\g^1)^2+(\g^2)^2>0, [\g]=\DC\} 
$$
is an open subset of  $W^{(\alpha+1)/2,2}(\R/\Z,\R^3)$.
\end{lem}
\bigskip
(Here,  $[\g]$ denotes the knot class represented by $\g$.  
In particular, $[\g]=\DC$ implies automatically that $\g|_{[0,1)}$ is injective.)
\begin{cor}\label{cor:openset}
The set $\Omega_\DC$ defined in Lemma \ref{lem:openset} is a smooth manifold modeled over
the Banach space $\sB:=W^{(\alpha+1)/2,2}(\R/\Z,\R^3)$.
\end{cor}
{\it Proof of Lemma \ref{lem:openset}.}\,
Fix 
  $\g\in\Omega_\DC$, and notice that $\g$
  is of class $C^{1,(\alpha/2)-1}(\R/\Z,\R^3)$ since $\alpha>2$ so that 
 the Morrey-type embedding holds; see \eqref{embedding}. In particular, there is
a constant $c_\g>0$ such that 
$
\min\left\{|\g'|,\sqrt{(\g^1)^2+(\g^2)^2}\right\}\ge c_\g$ on $[0,1]$.
Thus, for every $h\in W^{(\alpha+1)/2,2}(\R/\Z,\R^3)$ we find by means of
\eqref{embeddingineq}
\begin{align*}
\min\left\{|(\g+h)'|,\sqrt{(\g^1+h^1)^2+(\g^2+h^2)^2}\right\} & \ge c_\g-\|h\|_{C^{1,(\alpha/2)-1}}\\
 & \overset{\eqref{embeddingineq}}{\ge}
c_\g-C_E\|h\|_{W^{(\alpha+1)/2,2}}\ge \frac 12 c_\g>0,
\end{align*}
if $\|h\|_{W^{(\alpha+1)/2,2}}\le c_\g/(2C_E)$,
where $C_E=C_E(1,3)$ is the constant in the embedding inequality \eqref{embeddingineq}
in ambient space dimension $n=3$. 
According to the stability of the isotopy class under $C^1$-perturbations (see, e.g.
\cite{reiter_2005} or \cite{blatt_2009a}) there exists some $\eps_\g>0$ such that all 
 curves $\xi\in B_{\eps_\g}(\g)\subset C^1(\R/\Z,\R^3)$ are ambient isotopic to $\g$.
This implies that for any $h\in W^{(\alpha+1)/2,2}(\R/\Z,\R^3)$ with
$\|h\|_{W^{(\alpha+1)/2,2}}\le \eps_\g/C_E$ we have $\g+h\in B_{\eps_\g}(\g)\subset C^1(\R/\Z,\R^3)$,
so that $[\g+h]=\DC$. Setting $\delta:=\min\{\eps_\g,c_\g/2\}/C_E$ we conclude that
the open 
ball $B_\delta(\g)\subset W^{(\alpha+1)/2,2}(\R/\Z,\R^3)$ is actually contained
in $\Omega_\DC.$
\qed

Since we are going to look at symmetric knots under rotations with a fixed angle we are led
to consider the finite cyclic group $\Z/(m\Z)$, for which we recall its definition.
\begin{defin}\label{def:cyclic}
For $m\in\Z$ with $|m|\ge 2$ let
$G:=\Z/(m\Z)$ be the subgroup of $(\Z,+)$ consisting of the equivalence classes
$[z]$ determined by the equivalence relation 
$$
\textnormal{$z_1,z_2\in\Z$ are equivalent denoted by $ z_1\sim z_2$}\quad\Longleftrightarrow\quad
z_1=z_2+km \quad\textnormal{for some $k\in\Z$}.
$$
\end{defin}
The group $(G,+)$ forms a group with $m$ elements, where the addition is defined
as
$
[z_1]+[z_2]=[z_1+z_2]
$
which is well-defined since it does not depend on the choice of representatives. 

As we deal with parametrized curves we need to adjust rotations in space by 
appropriate parameter shifts in the domain. To be precise we establish in the following lemma
a set of group actions of $G$ (depending on an additional integer parameter)
on the Banach manifold $\Omega_\DC$ for any given knot class
$\DC$.  Here, and also later, we use the notation
\begin{equation}\label{rotnot}
\opRot(\beta):=\left(\begin{array}{ccc}
\cos\beta & -\sin\beta & 0\\
\sin\beta & \cos\beta & 0\\
0 & 0 & 1\end{array}\right)\in SO(3)
\end{equation}
for the rotation matrix about the $z$-axis (with respect to the standard 
basis of
$\R^3$), and we write, more generally, $\opRot(\beta,v)$ for a
rotation about an arbitrary axis $v$ with rotational angle $\beta$. Notice 
that in that case $v$ does not necessarily contain the origin.
\begin{lem}\label{lem:groupaction}
Let $\DC$ be an arbitrary tame knot class, and fix $\alpha\in (2,3)$, $k,m\in\Z$, 
and let $G:=\Z/(m\Z)$.
Then $G$ acts on $\Omega_\DC$ via 
the mapping
\begin{eqnarray*}
\tau^k:  G\times\Omega_\DC &\longrightarrow & \Omega_\DC\notag\\
  (g,\g) &\longmapsto  & \tau^k_g(\g)
\end{eqnarray*}
defined as
\begin{equation}\label{tauk}
\tau^k_g(\g)(t):=D_g\g(t+\tfrac km \cdot l_g)\Fo t\in \R/\Z,
\end{equation}
where $D_g=\opRot(2\pi l_g/m)\in SO(3)$ and 
$l_g\in\Z$ is a representative of $g\in G.$
Moreover, $\Omega_\DC$ becomes a smooth $G$-manifold under this action.
\end{lem}
\begin{rem}\label{welldefinedaction}
As $\g$ is $1$-periodic,  $\tau^k_g$ in \eqref{tauk} is obviously 
well-defined since it does not depend  on the
choice of representative $l_g$, since any other representative differs from $l_g$
only by an integer multiple of $m$.
\end{rem}
{\it Proof of Lemma \ref{lem:groupaction}}.\,
Since a rotation in the ambient space and a parameter shift does not change
the Sobolev-Slobodetckij norm we find that $\tau^k_g(\g)\in W^{(\alpha+1)/2,2}
(\R/\Z,\R^3)$ for any $\g\in W^{(\alpha+1)/2,2}
(\R/\Z,\R^3)$. Moreover, 
\begin{multline*}
\min\left\{
|\tau^k_g(\g)'(t)|,\sqrt{((\tau^k_g(\g))^1(t))^2+((\tau^k_g(\g))^2(t))^2}\right\}
\\=\min\left\{|\g'(t+\tfrac km \cdot l_g)|,\sqrt{(\g^1(t+\tfrac km \cdot l_g))^2+
(\g^2(t+\tfrac km \cdot l_g))^2}\right\} >0\Foa t\in\R/\Z.
\end{multline*}
A parameter shift combined with a rotation in ambient space does not change the knot
type, that is,  $[\tau^k_g(\g)]=\DC,$ so that $\tau^k_g(\g)\in\Omega_\DC$ for any
$\g\in\Omega_\DC$. We need to check that $\tau^k$ is a representation of $G$ on $\Omega_\DC$;
cf. Definition \ref{def:groupact}. Indeed,
for $g,h\in G$ we may choose the representative $l_{g+h}=l_g+l_h$ as a representative for the
group element $g+h\in G$, so that
\begin{align*}
\tau^k_{g+h}(\g)(t) & = D_{g+h} \g\li t+\tfrac km l_{g+h} \ri 
 = D_g D_h \g\li t+\tfrac km (l_g+l_h) \ri  \\
& = D_g  \li D_h \g\li \cdot + \tfrac km l_h \ri \ri\li t+\tfrac km l_g \ri 
 = D_g \tau^k_h(\g)\li t+\tfrac km l_g \ri =\tau^k_g\li \tau^k_h(\g) \ri (t).
\end{align*}
Finally, one has smoothness of $\tau^k_g:\Omega_\DC\to\Omega_\DC$ 
for any fixed
$g\in G$ since $\tau^k_g$ is linear:
\begin{align*}
\tau^k_g(\lambda \g + \eta)(t)&=
D_g  (\lambda\g+\eta)\li t+\tfrac km l_g \ri \\
& = \lambda D_g  \g \li t+\tfrac km l_g \ri +D_g \eta \li t+\tfrac km l_g \ri =\lambda\tau^k_g(\g)+
\tau^k_g(\eta)
\end{align*}
for all $\g,\eta\in W^{(\alpha+1)/2,2}(\R/\Z,\R^3)$ and $\lambda\in\R.$
In particular, for the differential of $\tau^k_g$ at $\gamma\in
\Omega_\mathcal{K}$ one simply has 
$$
(d\tau^k_g)_\g[\eta]=\tau^k_g(\eta)\Foa \eta \in W^{(\alpha+1)/2,2}(\R/\Z,\R^3),
$$
which implies according to Definition \ref{def:groupact}
that $\Omega_\DC$ is a smooth $G$-manifold, since $\tau^k_g$ is an isomorphism with inverse mapping
$$
(\tau^k_g)^{-1}(\g):=D_{-g} \g\li t+\tfrac km l_{-g}\ri,
$$
where $l_{-g}$ is a representative of the group element $-g\in G$  (with $g+(-g)=e:=[0]\in G$),
e.g. $l_{-g}=-l_g$. 
\qed

For technical reasons we will have to reparametrize to arclength later in our existence proof
of minimizers in the $G$-symmetric subset, and therefore
we need to understand what kind of
symmetry the arclength parametrization inherits from a symmetric curve.
\begin{lem}\label{lem:inherit-symmetry}
Let $m,k\in\Z$,  $G=\Z/(m\Z)$, and
 $\g:\R/\Z\to\R^3$ be an absolutely continuous curve
 with $|\g'|>0$ a.e. and with
length $\sL(\g)=L\in (0,\infty)$, such that 
for $g=[l_g]\in G$ the identity 
$\tau^k_g(\g)=\g$ holds with
$\tau^k_g$ as in \eqref{tauk}. Then
the corresponding
arclength parametrization $\Gamma\in C^{0,1}(\R/(L\Z),\R^3)$ satisfies
\begin{equation}\label{inherit-symmetry}
D_g  \Gamma\li s+\tfrac km l_gL\ri =\Gamma(s)\Foa s\in [0,L).
\end{equation}
\end{lem}
Since arclength reparametrizations of curves in  $W^{(\alpha +1)/2,2}$ inherit the same
regularity as shown in the appendix in  Lemma \ref{lem:arclength-seminorm} we immediately
infer the following corollary.
\begin{cor}\label{cor:inherit-symmetry}
Let $m,k\in\Z$ and $G=\Z/(m\Z)$ and let
$\DC$ be any knot class, and $\Omega_\DC$ be the Banach manifold defined in Lemma \ref{lem:openset}
with $G$-symmetric subset $\Sigma_\DC^k$ with respect to the group action given by $\tau^k$ defined in \eqref{tauk}.
Then, if
$\g\in \Sigma^k_\DC$ with length $\sL(\g)=1$, its arclength parametrization 
$\Gamma:\R/\Z\to\R^3$ is contained in $\Sigma^k_\DC$ as well.
\end{cor}

{\it Proof of Lemma \ref{lem:inherit-symmetry}.}\,
Differentiating the relation $\tau_g^k(\g)=\g$ with respect to $t$ one obtains
$D_g \g'\li t+\frac km l_g \ri =\g'(t)$ for a.e. $t\in\R/\Z$.
Since $D_g\in SO(3)$ we find that $|\g'|$ is not only $1$-periodic but also  $kl_g/m$-periodic, so
that we can calculate 
for the arclength parameter
\begin{align}
s\left(\tfrac km l_g\right)=\int_0^{\frac km l_g}|\g'(t)|\,dt & =\frac 1m \int_0^{kl_g}
|\g'(t)|\,dt\notag\\
& = \frac km l_g\int_0^1|\g'(t)|\,dt= L\frac km l_g,\label{eq:bgl-anfang}
\end{align}
and therefore,
\begin{align}
s\left(\tfrac{k}m l_g+t\right) & =\int_0^{\frac km l_g +t}|\g'(\tau)|\,d\tau
=\int_0^{\frac km l_g}|\g'(\tau)|\,d\tau +\int_{\frac km l_g}^{\frac km l_g+t}
|\g'(\tau)|\,d\tau\notag\\
&\overset{\eqref{eq:bgl-anfang}}{=}L\frac km l_g +\int_0^t |\g'(\tau)|\,d\tau=
L\frac km l_g +
s(t).\label{eq:bgl-allgemein}
\end{align}
With $\Gamma(s(t))=\g(t) $ for all $t\in\R/\Z$ we infer from this
 by definition of the group action \eqref{tauk}
 \begin{align*}
 \Gamma(s(t))=\g(t)=\tau^k_g(\g)(t)\overset{\eqref{tauk}}{=}
 D_g\g\left(t+\tfrac km l_g\right)=
 D_g\Gamma(s\left( t+\tfrac km l_g \right))
 \overset{\eqref{eq:bgl-allgemein}}{=} D_g\Gamma\left(
 s(t) +\tfrac km l_g L\right).
 \end{align*}
\qed

Now we turn our attention to torus knots.
For relatively prime integers $a,b\in\Z\setminus\{0,\pm 1\}$ and 
some fixed  $\rho\in (0,1)$  the curve
\begin{equation}\label{toruscurve}
\g_\rho(t):=\opRot\li 2\pi at \ri  \left(\begin{array}{c}
1+\rho\cos(2\pi bt)\\
0\\
\rho\sin(2\pi bt)\end{array}\right)=\left(\begin{array}{c}
\cos(2\pi at)(1+\rho\cos(2\pi bt)\\
\sin(2\pi at)(1+\rho\cos(2\pi bt)\\
\rho\sin(2\pi bt)\end{array}\right)\Fo t\in\R/\Z 
\end{equation}
is a smooth representative of the torus knot class $\TP(a,b)$. According to \cite[Theorem 3.29]{burde-zieschang_1985}
one has $\TP(a,b)=\TP(b,a)=\TP(-a,-b)=\TP(-b,-a)$.
We can use the particular representative $\g_\rho$ defined in \eqref{toruscurve}
to show that the $G$-symmetric subset of the Banach manifold $\Omega_{\TP(a,b)}$ with respect to
the group action \eqref{tauk} is not empty.
\begin{lem}\label{lem:symmetric-subset}
Let $\alpha\in (2,3)$, $a,b\in\Z\setminus\{0,\pm 1\}$ relatively prime,  
let $m\in\N$, $m>1$, 
divide $a$ or $b$, and let $G=\Z/(m\Z)$. Then the following is true:
For any $k\in\Z\setminus\{0\} $ with 
\begin{equation}\label{k}
\begin{cases}
[ak+1]=e=[0]\in G & \textnormal{if $m|b$}\\
[bk+1]=e=[0]\in G & \textnormal{if $m|a$}
\end{cases}
\end{equation}
 one has a nonempty $G$-symmetric subset 
 $$
 \Sigma^{m}_{a,b}:=\{\g\in\Omega_{\TP(a,b)}:\tau_g(\g)=\g\Foa g\in G\}, 
$$
where  $\tau_g$ is defined  in \eqref{tauk}.
\end{lem}
\begin{proof}
It suffices to treat the case $m|b$.
In Lemma \ref{lem:restklasse} in the appendix we show that such 
$k\in\Z\setminus\{0\}$ with \eqref{k} do exist, furthermore, $k$
is unique modulo $m$.
Taking $\g_\rho$ as in \eqref{toruscurve} as a smooth and regular representative for $\TP(a,b)$ that avoids the $z$-axis,
we find that $\g_\rho\in\Omega_{\TP(a,b)}$, and we directly compute
\begin{align*}
\tau_g(\g_\rho)(t) & = D_g \g_\rho\li t+\tfrac km l_g \ri \\
&=\opRot\li 2\pi l_g/m \ri  
\opRot\li 2\pi a\li t+\tfrac km l_g \ri  \ri  
\left(\begin{array}{c}
1+\rho\cos\li 2\pi b\li t+\frac km l_g \ri  \ri \\
0\\
\rho\sin\li 2\pi b\li t+\frac km l_g \ri  \ri \end{array}\right)\\
& = 
\opRot\li 2\pi at + 2\pi(ak+1)l_g/m \ri   \left(\begin{array}{c}
1+\rho\cos(2\pi bt)\\
0\\
\rho\sin(2\pi bt)\end{array}\right)=\g_\rho(t),
\end{align*}
where we used \eqref{k} in the argument of the last rotation.
Hence, $\g_\rho\in\Sigma^{m}_{a,b}.$
\end{proof}
Now we are ready to prove the existence of symmetric minimizers for the scaled
O'Hara energy defined in \eqref{scaledenergy} in the introduction.
Notice that since $E_\alpha$ 
is continuously differentiable on the space of regular
curves (see Remark \ref{rem:C1}),  so is $S_\alpha$
since the length functional is continuously differentiable,
 even in the class of regular curves of class
 $W^{1,1}(\R/\Z,\R^3)$, and hence in particular on the Banach manifold
$\Omega_\DC$ for any (tame) knot class $\DC$.
\begin{thm}\label{thm:existence}
Let 
$\alpha\in (2,3),$ $a,b\in\Z\setminus\{0,\pm 1\}$ relatively prime, and let 
$m\in\N$, $m>1$, divide $a$ or $b$. Then for any 
 $k\in\Z\setminus\{0\}$ satisfying condition \eqref{k} of Lemma \ref{lem:symmetric-subset}
there exists an arclength parametrized curve $\Gamma^{m}_\textnormal{min}\in\Sigma^{m}_{a,b}\subset\Omega_{\TP(a,b)}$
such that 
\begin{equation}\label{existence}
S_\alpha\li \Gamma^{m}_\textnormal{min} \ri =\inf_{\Sigma_{a,b}^{m}}S_\alpha.
\end{equation}
Here  $\Sigma^{m}_{a,b}$ is the nonempty $G$-symmetric subset of $\Omega_{\TP(a,b)}$, $G=\Z/m\Z$, 
with respect to the group action of $\tau$ defined in \eqref{tauk}; see Lemma \ref{lem:symmetric-subset}. 
\end{thm}
\begin{proof}
Without loss of generality we may assume $m|b$, the  case $m|a$ can be treated
analogously. According to Lemma \ref{lem:symmetric-subset} we have 
$\Sigma_{a,b}^{m}\not=\emptyset.$   The energy is finite on this set (see
part (ii) of Theorem \ref{thm:blatt}), so we find a minimizing sequence $(\g_i)_i\subset
\Sigma^{m}_{a,b}$ with
\begin{equation}\label{minimizingseq}
\lim_{i\to\infty}S_\alpha(\g_i)=\inf_{\Sigma^{m}_{a,b}}S_\alpha\in [0,\infty).
\end{equation}
Since $S_\alpha$ is scale-invariant we may assume, in addition, that $\sL(\g_i)=1$
for all $i\in\N$ (simply
by scaling the $\g_i$ with scaling factor $\sL(\g_i)^{-1} $ if necessary).
In addition, by translations in the $z$-direction (thus keeping
the symmetry), we may also assume that all $\g_i$ intersect the $x$-$y$-plane.

By \eqref{minimizingseq},
$$
S_\alpha(\g_i)=E_\alpha(\g_i)\le C\Foa i\in\N,
$$
where $C$ is a constant independent of $i$. Since $\sL(\g_i)=1$ for all $i\in\N$,
the corresponding arclength parametrizations $\Gamma_i$ all have the common
domain $\R/\Z$ and $E_\alpha(\Gamma_i)=E_\alpha(\gamma_i)$ for all
$i\in\N$. Moreover, according to \eqref{energy-estimate1} in part (i) of
Theorem \ref{thm:blatt} these arclength parametrizations
are all of class $W^{(\alpha+1)/2,2}(\R/\Z,\R^3)$ satisfying
\begin{equation}\label{unifbound}
 [\Gamma_i']_{{(\alpha-1)/2,2}}\le 4^4\cdot 2^{2-2\alpha}C\Foa i\in\N.
\end{equation}
Since all $\Gamma_i$ have length $1$, each $\Gamma_i$ is contained
in a closed ball $B_i\subset\R^3$ of radius\footnote{or even in a closed
ball of radius $1/4$; see the short argument in \cite{nitsche_1971}.} $1/2$.
All these closed balls $B_i$ must intersect the $x$-$y$-plane since 
$\Gamma_i$ does for each $i\in\N.$ In addition, by symmetry the
$B_i$ also
intersect the $z$-axis.
Indeed, the orbit of a point $x\in\Gamma_i$ under the action 
of $G$ lies in a hyperplane orthogonal to the $z$-axis, 
and the convex hull of this orbit is an
$m$-gon in that hyperplane that intersects the $z$-axis and is 
contained in $B_i$,  so that $B_i$ itself intersects the $z$-axis as well.
Therefore all 
$B_i$ and thus all $\Gamma_i(\R/\Z)$ are contained 
in a cube of edge length $4$ centered at
the origin, so that 
$$
\|\Gamma_i\|_{L^\infty}\le\sqrt{8}\Foa i\in\N.
$$
Combining this with \eqref{unifbound} and the identity
$|\Gamma_i'|\equiv 1$ for all $i\in\N$ 
we arrive at 
$$
\|\Gamma_i\|_{W^{(\alpha+1)/2,2}}\le C_1\Foa i\in\N,
$$
where $C_1$ is independent of $i$. Together with the embedding inequality \eqref{embeddingineq}
we arrive at a uniform $C^{1,(\alpha/2)-1}$-bound 
$$
\|\Gamma_i\|_{C^{1,(\alpha/2)-1}}\le C_EC_1\Foa i\in\N.
$$
By the Arzela-Ascoli compactness theorem we find a subsequence (again denoted
by $\Gamma_i$), which converges strongly in $C^1$ to a limit curve
$\Gamma\in C^{1,\mu}$ for all $\mu\in (0,(\alpha/2)-1).$ This convergence implies
in particular that $|\Gamma'|\equiv 1$. 
We have shown in Lemma \ref{lem:continuous} 
that $E_\alpha$ is lower semicontinuous
even with respect to pointwise convergence, which
implies that $E_\alpha(\Gamma)\le\liminf_{i\to\infty}E_\alpha(\Gamma_i)\le C$.
According to Part (i) of Theorem \ref{thm:blatt} the limit curve 
$\Gamma$ is of class $W^{(\alpha+1)/2,2}$ and injective.
Now, the isotopy stability under $C^1$-convergence mentioned before (see \cite{reiter_2005} or \cite{blatt_2009a}) gives
$[\Gamma]=[\Gamma_i]=\TP(a,b)$ for all $i\in\N$. 
In order to establish the symmetry of $\Gamma$ 
we use Corollary \ref{cor:inherit-symmetry}, which implies that
$$
D_g \Gamma_i\li s+\tfrac km l_g \ri =\Gamma_i(s) \Foa s\in [0,1),\,i\in\N.
$$
Taking the limit $i\to\infty$ in this relation (for the subsequence $\Gamma_i$
converging in $C^1$ to $\Gamma$) implies 
\begin{equation}\label{lim-inherit-symmetry}
D_g \Gamma\li s+\tfrac km l_g \ri =\Gamma(s) \Foa s\in [0,1),
\end{equation}
and hence $\tau_g(\Gamma)=\Gamma$ for all $g\in G.$
 Now, if there was some parameter $s\in\R/\Z$ such that
$(\Gamma^1(s))^2+(\Gamma^2(s))^2=0$ we could apply 
\eqref{lim-inherit-symmetry} to find that
\begin{equation}\label{contra1}
\Gamma\li s+\tfrac km l_g \ri =\Gamma(s) \Foa g\in G,
\end{equation}
since the rotation $D_g=\opRot\li 2\pi l_g/m \ri $ about the $z$-axis and
hence also its inverse leave
every point on the $z$-axis fixed.  But \eqref{contra1} contradicts
the injectivity of $\Gamma$ since $k\not=0$ and $g\in G$ may be 
chosen to be non-trivial. Thus we have shown that 
$\Gamma\in \Sigma^{m}_{a,b}\subset\Omega_{\TP(a,b)}.$
This together with the lower semicontinuity of $E_\alpha$ established in
Lemma \ref{lem:continuous} finally implies minimality for
$\Gamma^{m}_\textnormal{min}:=\Gamma$ because
$$
\inf_{\Sigma^{m}_{a,b}}S_\alpha\le S_\alpha(\Gamma)=
E_\alpha(\Gamma)\le\liminf_{i\to\infty}E_\alpha(\Gamma_i)=\lim_{i\to\infty}S_\alpha
(\Gamma_i)=\inf_{\Sigma^{m}_{a,b}}S_\alpha.
$$
\end{proof}

Now we can convince ourselves that these symmetric minimizing torus knots are
all critical for the scaled energy functional $S_\alpha$ on all of $
\Omega_{\TP(a,b)}$.
\begin{cor}\label{cor:criticaltorus}
Any of the minimizing torus knots 
$\Gamma^{m}_{\textnormal{min}}\in
\Sigma^{m}_{a,b}$
found in Theorem \ref{thm:existence}
are critical points of the scaled energy $S_\alpha=\sL^{\alpha-2}E_\alpha$ and
therefore of class $C^\infty(\R/\Z,\R^3).$
\end{cor}
\begin{proof}
We have seen in Corollary
\ref{cor:openset} that $\Omega_{\TP(a,b)}$ is a smooth
manifold modeled over the Banach space
$W^{(\alpha+1)/2,2}(\R/\Z,\R^3).$ In addition, according to Lemma \ref{lem:groupaction} $\Omega_{\TP(a,b)}$ is even a smooth $G$-manifold under the action of the
finite group $G:=\Z/(m\Z)$ for $m\in\N\setminus\{1\}$.
Moreover, the scaled energy
$S_\alpha=\sL^{\alpha-2}E_\alpha$ is  of class $C^1$ on an open subset
of the Banach space $W^{(\alpha+1)/2,2}(\R/\Z,\R^3)$ containing
$\Omega_{\TP(a,b)}$ as mentioned in
Remark \ref{rem:C1}, and $S_\alpha$ is invariant under the action
of $\tau$ since rotations in the ambient space and parameter shifts
obviously do not alter the energy value; see Remark \ref{rem:generalprop}. 
Since
the $\Gamma^m_{\textnormal{min}}$ minimize $S_\alpha$ in $\Sigma^{m}_{a,b}$, they   
are $S_\alpha|_{\Sigma^{m}_{a,b}}$-critical 
and therefore, according to Palais' Theorem \ref{thm:palais}, the 
$\Gamma^m_{\textnormal{min}}$ are also critical for $S_\alpha$ on the full
domain $\Omega_{\TP(a,b)}$.  The smoothness now follows by
the regularity theorem of Blatt and Reiter mentioned in Remark \ref{rem:C1}.
\end{proof}
In order to show that there are \emph{ at least two} $S_\alpha$-critical knots
in every non-trivial torus knot class $\TP(a,b)$ we recall the definition of
periodicity of knots from  \cite[p. 256]{burde-zieschang_1985} (see also
\cite[Definition 8.3]{livingston_1995}):
Any  curve $\g\in C^0(\R/\Z,\R^3)$ being injective on $[0,1)$ 
that  does not intersect the $z$-axis, and
 for which there is an integer 
$q\in\N\setminus\{1\}$ such that
$$
\opRot\li 2\pi/q \ri \g(\R/\Z)=\g(\R/\Z)
$$
has \emph{period $q$}, or is \emph{$q$-periodic}.

For torus knots the possible periods are known; see \cite[Proposition
14.27]{burde-zieschang_1985}:
\begin{thm}\label{thm:periods}
If $q\in\N\setminus\{1\}$ is a period of a curve $\g\in C^0(\R/\Z,\R^3)$
with $[\g]=\TP(a,b)$ for relatively prime
integers $a,b\in\Z\setminus\{0,\pm 1\}$, then $q|a$ or $q|b$. Conversely,
if $q\in\N\setminus\{1\}$ divides $a$ or $b$, then there is a representative
$\g\in C^0(\R/\Z,\R^3)$ such that $q$ is a period of $\g$.
\end{thm}
 This result allows us to prove that there are at least two 
 $S_\alpha$-critical knots in every torus knot class, which is
 our central result, Theorem \ref{thm:twocritical} mentioned
 in the introduction.

 {\it Proof of Theorem \ref{thm:twocritical}.}
 For each $m\in\N\setminus\{1\}$ dividing $a$ or $b$, and for each $k\in\Z$
 satisfying \eqref{k} Theorem \ref{thm:existence} in connection with
 Corollary \ref{cor:criticaltorus} gives us at least one arclength
 parametrized curve 
 $$
 \Gamma^{m}_\textnormal{min}\in 
 \Sigma^{m}_{a,b}\cap C^\infty(\R/\Z,\R^3)
 $$ 
 that is
 $S_\alpha$-critical. 
Choosing $m_1:=a$ and $k_1$ such that $k_1$ satisfies \eqref{k} 
for $m=m_1$, as well as $m_2:=b$ and $k_2$ satisfying \eqref{k} for
$m=m_2$, we obtain two curves 
$$
\Gamma_1:=\Gamma^{a}_\textnormal{min}
\in\Sigma^{a}_{a,b}\cap C^\infty(\R/\Z,\R^3)\AND
\Gamma_2:=\Gamma^{b}_\textnormal{min}
\in\Sigma^{b}_{a,b}\cap C^\infty(\R/\Z,\R^3)
$$
with
\begin{align}
D_g \Gamma_1(\R/\Z)&=\Gamma_1(\R/\Z)\Foa g\in\Z/(a\Z),\label{symma}\\
D_h \Gamma_2(\R/\Z)&=\Gamma_2(\R/\Z)\Foa h\in\Z/(b\Z)\label{symmb}
\end{align}
by means of \eqref{inherit-symmetry} with $L=1$ for $m=m_1=a$, and 
for $m=m_2=b$, 
respectively.

Any isometry $I:\R^3\to\R^3$ can be written as $I(x)=Ox+\xi$, $x\in\R^3$,
for some
orthogonal matrix $O\in O(3)$ and some vector $\xi\in\R^3.$ Since the orthogonal group $O(3)$ is the semidirect product of $SO(3)$ and $O(1)$ 
\cite[p.50]{fischer_2013}, we can write $O=S R$  for some 
rotation $R\in SO(3)$ and some $S\in O(1)$, and the latter may be a 
reflection
across one two-dimensional subspace $E\subset\R^3$, or else 
$S$ is  the identity mapping. But if $S$ is a 
reflection and we assume that 
\begin{equation}\label{contra-symm}
I\circ\Gamma_1(\R/\Z)
=\Gamma_2(\R/\Z),
\end{equation} then (since translations and rotations
do \emph{not} alter the knot class)
\begin{equation}\label{contra-symmA}
\TP(a,b)=[\Gamma_2]=[I\circ\Gamma_1]=[O \Gamma_1]=[S  R 
\Gamma_1]\not= [R \Gamma_1]=[\Gamma_1]=\TP(a,b),
\end{equation}
which is a contradiction. To justify the inequality in 
\eqref{contra-symmA} note that
according to \cite[Theorem 3.29]{burde-zieschang_1985} 
the torus knot class $\TP(a,b)$ is \emph{not} amphichiral, i.e., the reflection  $S \g$ of any curve $\g$ with $[\g]=
\TP(a,b)$ at some two-dimensional subspace $E\subset\R^3$
 would represent the different torus knot class
$\TP(a,-b)\not= \TP(a,b)$; see \cite[Prop. 3.27]{burde-zieschang_1985}.
So, the assumption \eqref{contra-symm}
necessarily leads to the representation $I(x)=Rx+\xi$, $x\in\R^3$,
for some rotation
$R\in SO(3)$ (about some axis through the origin) and some translational
vector $\xi\in\R^3$.

This together with \eqref{symma}, \eqref{symmb}, 
and the fact that
$\Gamma_1,\Gamma_2\in\Omega_{\TP(a,b)}$ both do not intersect the 
$z$-axis,
 implies under the assumption \eqref{contra-symm} that 
 $\Gamma_2=I\circ\Gamma_1$ is $a$-periodic 
with respect to the axis $I(\R e_3)$ in addition to being $b$-periodic
with respect to the $z$-axis; see also Lemma \ref{lem:rotationsets}
in the appendix.
Theorem \ref{thm:gruenbaum-gilsbach} below then 
implies 
that the axis $I(\R e_3)$ coincides with the $z$-axis 
since the two rotational axes must necessarily intersect, and
if there were only one intersection point of these axes, then the
two different rotational angles $2\pi/a\not= 2\pi/b$ would lead
to a nonempty intersection of $\Gamma_2$ with one of the rotational
axes contradicting the periodicity of $\Gamma_2$; see Part (1)(iii)
of
Theorem
\ref{thm:gruenbaum-gilsbach}.

Since the $z$-axis equals its image under the isometry
$I$ we can infer in particular 
that the vector $\xi=R0+\xi=I(0)$ is contained in the $z$-axis;
hence $\xi=(0,0,\xi_3)$. Therefore, we find some $\lambda\in\R$ such
that the point $I(e_3)=Re_3+\xi_3e_3$
which is also contained in the $z$-axis may be written as $I(e_3)=
\lambda e_3$ so that $Re_3=(\lambda-\xi_3)e_3=:\mu e_3.$ 
So $\mu$
is a real eigenvalue for the rotation $R\in SO(3)$; hence
$\mu$ is either $+1$ or $-1$. In the first case $e_3$ belongs to
the fixed point set of $R$ which implies that $R$ is a rotation
about the $z$-axis. If $\mu=-1$, on the other hand, $R$ is a rotation
about an axis perpendicular to the $z$-axis with the rotational angle $\pi$.

In both cases $R$ commutes with $D_h$ on $\Gamma_1$, see 
Lemma \ref{lem:commut_rotation}, 
which itself is a rotation about the $z$-axis, 
so that we infer  (omitting the domain $\R/\Z$ in each term)
$$
R\Gamma_1+\xi \overset{\eqref{contra-symm}}{=}\Gamma_2 \overset{\eqref{symmb}}{=} D_h\Gamma_2\overset{\eqref{contra-symm}}{=}D_h(R\Gamma_1+
\xi)=D_h R \Gamma_1+D_h(\xi)=D_h R \Gamma_1+\xi = R D_h \Gamma_1+\xi,
$$
where the second to last
equality  is due to the fact that
$\xi$ is contained in the $z$-axis. This leads to
$$
R \Gamma_1(\R/\Z)=R D_h\Gamma_1(\R/\Z),
$$
which implies a second symmetry of $\Gamma_1$ in addition to
\eqref{symma}:
\begin{equation}\label{symmaa}
D_h \Gamma_1(\R/\Z)=\Gamma_1(\R/\Z)\Foa h\in \Z/(b\Z).
\end{equation}
Now choosing 
$g=[1]\in\Z/(a\Z)$ and $h=[1]\in\Z/(b\Z)$ we find another period of 
$\Gamma_1$ as follows 
(again omitting the domain $\R/\Z$ in each term):
\begin{equation}\label{RRrot}    
\Gamma_1  
\overset{\eqref{symma}}{=}
D_g \Gamma_1
\overset{\eqref{symmaa}}{=}
D_g D_h \Gamma_1=
\opRot\left(\tfrac{2\pi}{a}\right) \opRot\left(
\tfrac{2\pi}{b}\right) \Gamma_1
=
\opRot\left(\tfrac{2\pi}{ab}\cdot(b+a)\right) \Gamma_1.
\end{equation}
The two integers, $(a+b)$ and $ab$, are relatively prime (see Lemma \ref{lem:gcd} in the appendix), so that $(a+b)$ is invertible modulo $ab$, which
means that we can find some integer $k\in\Z$ such that $k(a+b)\equiv 1\mod ab$. 
This implies by means of \eqref{RRrot} that
$$
\Gamma_1\overset{\eqref{RRrot}}{=}\left[\opRot\left(\tfrac{2\pi}{ab}\cdot (a+b)\right)\right]^k \Gamma_1
=
\opRot\left(\tfrac{2\pi}{ab}\cdot k(a+b)\right) \Gamma_1
=
\opRot\left(\tfrac{2\pi}{ab}\right) \Gamma_1.
$$
In other words, $\Gamma_1$ is $(ab)$-periodic, which contradicts
Theorem \ref{thm:periods}, since $ab$ divides neither $a$ nor $b$.
This is the final contradiction and concludes the proof of the theorem.
\qed

 Essential for the previous proof is the following  result on 
 possible rotational symmetries of general non-trivial knots. Most of these 
facts can also be extracted from Gr\"unbaum and Shephard's classification of possible symmetry groups of knots \cite{gruenbaum-shephard_1985} in combination with their characterization of finite subgroups of $O(3)$ in \cite{gruenbaum-shephard_1981}. 
 Here we present a purely geometrical approach, adding information about possible periods of a knot. 

 \begin{thm}[Rotational symmetries of knots]\label{thm:gruenbaum-gilsbach}
 If a non-trivial tame knot $\Gamma$ has a rotational 
 symmetry about an axis $v$ with angle $\varphi\in (-\pi,\pi]$ 
 and 
 $\Gamma\cap v\neq \emptyset$, then  $\varphi=\pi$. 
 If $\Gamma$ has two axes $v_1$ and $v_2$  of rotational symmetry 
 with respect to rotation angles $\varphi_1=\frac{2\pi}{a_1},
 \varphi_2=\frac{2\pi}{a_2}$ for some integers $a_1,a_2\ge 2$, 
 then
 $v_1\cap v_2\neq\emptyset$. 
 
 Furthermore, if $v_1\cap v_2=\{p\}$ for some $p\in\R^3$, the following holds.
 \begin{enumerate}
  \item For $\varphi_1\neq\varphi_2$ we have 
	\begin{enumerate}
	\item[\rm (i)]
	$v_1\perp v_2$;
	\item[\rm (ii)]
	Either $a_1=2$ and $a_2\geq 3$, or 
	 vice versa; 
	\item[\rm (iii)]
	If $a_1=2$ in Part (ii) then
	$v_1\cap\Gamma \neq\emptyset$ and $v_2\cap\Gamma=\emptyset$.
	If $a_2=2$ in Part (ii) then
	$v_2\cap\Gamma \neq\emptyset$ and $v_1\cap\Gamma=\emptyset$.
	\end{enumerate}
  \item If $\varphi_1=\varphi_2$, then we have $\varphi_1=\varphi_2=\pi$.
 \end{enumerate}
\end{thm}

Before proving this theorem let us provide a slight generalization of
a result of Gr\"unbaum and Shephard \cite[Lemma 1]{gruenbaum-shephard_1985} whose paper actually motivated our purely geometric proof of
Theorem \ref{thm:gruenbaum-gilsbach}.
\begin{lem}\label{lem:number-symm-axes}
For $a\in\N$, $a\ge 3$, a knot cannot have more than one
axis of rotational symmetry with rotational angle $2\pi/a$.
\end{lem}
\begin{proof}
Assume that there are two 
axes $v$ and $w$ (not necessarily through the origin) of 
rotational symmetry for a knot $\Gamma\in\R^3$ with respect to
the rotational angle $\beta:=2\pi/a$ for some integer $a\ge 3.$ Fix a point
$x\in\Gamma$ and look at its orbit
$$
O_v:=\{x,x_1,x_2,\ldots,x_{a-1}\}\subset\Gamma
$$
under the action of the rotation $\opRot(\beta,v)$, i.e.,
$x_i:=\opRot(\beta i,v)x$ for $i=1,\ldots,a-1$, where
the symbol $\opRot(\beta,v)$ denotes the rotation about the
axis $v$ with angle $\beta$. The points  in $O_v$ are separated on
$\Gamma$ by subarcs of length $\mathscr{L}(\Gamma)/a$, and those
points form a regular $a$-gon spanning an affine 
plane $E_v$ perpendicular to the axis $v$ since $a\ge 3$. 
Let
$$
O_w:=\{x,\xi_1,\xi_2,\ldots,\xi_{a-1}\}\subset\Gamma
$$
be the corresponding orbit of $x$ under the rotation 
$\opRot(\beta,w)$, which also forms a regular $a$-gon spanning an
affine plane $E_w$
perpendicular  to the other axis $w$. The points in $O_w$
are separated on $\Gamma$ by subarcs  of length
$\mathscr{L}(\Gamma)/a$ as well, so that either
$x_i=\xi_i$, or $x_i=\xi_{a-1-i}$
for $i=1,\ldots,a-1$. 
In both cases the regular $a$-gons coincide, as well as the affine
planes
$E_v$ and $E_w$. Hence $v$ and $w$ are parallel,
and since both axes of rotational symmetry must intersect 
the midpoint of the $a$-gon
$$
(x+x_1+x_2+\cdots +x_{a-1})/a=
(x+\xi_1+\xi_2+\cdots +\xi_{a-1})/a,
$$
the axes $v$ and $w$ must coincide.
\end{proof}

{\it Proof of Theorem \ref{thm:gruenbaum-gilsbach}.}
To prove the first assertion, 
consider an angle $\varphi$ of an arbitrary rotation about 
an axis $v$ with $v\cap\Gamma\neq\emptyset$
with $\varphi\neq \pi$. Then we have $2\pi/|\varphi|>2$ arcs entering $x\in v\cap\Gamma$. But then $\Gamma$ is 
not embedded. Hence, if $\varphi\neq\pi$, we need to have $v\cap\Gamma=\emptyset$.\\
Now we consider the case of rotational symmetry about two different 
axes. 
We start by showing that a knot cannot have two rotational 
symmetry 
axes which are disjoint, no matter which angles are considered. 

To that extent, assume $\Gamma$ has two rotational symmetry axes 
$v_1, v_2$ with rotational angles $\varphi_1=2\pi/a_1$ and
$\varphi_2=2\pi/a_2$ for some integers $a_1,a_2\ge 2$,
such that 
$v_1\cap v_2=\emptyset$. 
If $a_1=a_2=2$ we argue as follows.
Consider the two parallel affine planes $E_1,E_2\subset\R^3$ such 
that $v_1\subset E_1$ and $v_2\subset E_2$, and $d:=\dist(E_1,E_2)>0$.
Then $\Gamma$ cannot be fully  contained in the closed infinite slab
$$
S:=\{x\in \R^3:\dist(x,E_i)\le d\Fo i=1,2\},
$$
since any point in $S$ gets mapped into
the exterior $\R^3\setminus S$ by at least one of the rotations
$\opRot(\pi,v_i)$, $i=1,2$. 
Now without loss of generality we may assume that $E_1$ and $E_2$ are 
parallel to the $x-y$-plane, i.e., 
$$
E_i:=\left\{y=(y^1,y^2,y^3)\in\R^3:y^3=R_i\right\}\Fo i=1,2
$$
with $R_1> R_2$, and we denote the curve
points  with the largest and the smallest $z$-coordinate by
$x_\textnormal{max}\in\Gamma$ and
$x_\textnormal{min}\in\Gamma$, respectively. We may assume without loss of
generality
that 
\begin{equation}\label{height}
\dist (x_\textnormal{max},S)\ge
\dist (x_\textnormal{min},S),
\end{equation}
and deduce for the point $x^*:=\opRot(\pi,v_2)
x_\textnormal{max}$
by means of \eqref{height}
the identity
$$
\dist\li x^*,E_2\ri =\dist\li x_\textnormal{max},E_2\ri =
\dist \li x_\textnormal{max},S\ri +d\overset{\eqref{height}}{\ge}
\dist \li x_\textnormal{min},S\ri +d
>\dist \li x_\textnormal{min},S\ri .
$$
Therefore, $x^*$ has a strictly smaller $z$-coordinate than
$x_\textnormal{min}$ since $x^*$ lies in $\R^3\setminus S$ below
the lower affine plane $E_2$, which contradicts the minimality of $x_\textnormal{min}$.
This settles the case $a_1=a_2=2.$

If, say $a_1\ge 3$ and $a_2\ge 2$, we can apply  repeatedly
Lemma \ref{lem:rotationsets} in the appendix to the set $M:=\Gamma$
and to the isometry $I$ defined as the rotation about $v_2$ with
respect to the rotational angle $\varphi_2=2\pi/a_2$. The
fact that
$I(\Gamma)=\Gamma$ because of the rotational symmetry of $\Gamma$
with respect to the rotation about $v_2$, together with
\eqref{Rrot} allows us to find
new symmetry axes for $\Gamma$ by rotating
$v_1$ about the other axis $v_2$.
That is, all axes 
\[
 v_1^i = \opRot\li \tfrac{2\pi\cdot i}{a_2},v_2\ri v_1,\quad i=0,...,a_2-1
\]
are axes of rotational symmetry for $\Gamma$ 
with rotational angle $\varphi_1=\frac{2\pi}{a_1}$, 
where, as before, the symbol $\opRot\li \beta,w\ri $ denotes
the rotation about an axis $w$ with rotational angle $\beta\in
\R$. Since $a_2\ge 2$ and $v_1\cap v_2=\emptyset$,
there are now at least two different
axes of rotational symmetry with respect to the angle $\varphi_1=
2\pi/a_1$, contradicting Lemma \ref{lem:number-symm-axes}.
Thus we have shown that $v_1\cap v_2\not=\emptyset.$

We will now assume that $v_1\cap v_2=\{p\}$ for some $p\in\R^3$. Without loss of generality we may restrict to the case
$p=0$ because of translational invariance of the
remaining claims.
The corresponding rotational angles are 
$\varphi_1=\frac{2\pi}{a_1}$ and $\varphi_2=\frac{2\pi}{a_2}$ 
for some integers $a_1,a_2\ge 2$. To prove Part (1) we 
take $a_1\neq a_2$ and consider the possible combinations of $a_1$ and $a_2$. \\
1. $a_1,a_2\geq 3$.\\
In this case both rotational angles are contained in $(0,\pi)$ so that
the first part of the theorem implies that $\Gamma$ is disjoint from
both axes $v_1$ and $v_2$. 
As before, we may construct copies of $v_1$ such that $\Gamma$ is rotational symmetric with respect to
the axis $v_1$, as well as to its copies
\[
 v_1^i = \opRot\li \tfrac{2\pi\cdot i}{a_2},v_2\ri v_1,\quad i=0,...,a_2-1.
\]
In other words, 
all these lines are axes of rotational symmetry for $\Gamma$
with the same
rotational angle $\varphi_1=\frac{2\pi}{a_1}$ with $a_1\ge 3$, and
there are at least two of those since $a_2\ge 3$, contradicting
Lemma \ref{lem:number-symm-axes}. 
Thus,  either $a_1=2$ and $a_2\ge 3$, or $a_2=2$ and $a_1\ge 3$ which proves Part (1)(ii).
Furthermore, the 
presented argument implies Part (2).

2. $a_1\geq 3$, $a_2=2$, (the case $a_1=2$ and $a_2\ge 3$ can be
treated analogously).\\
In this case, we will have to take into account the angle 
$\ang(v_1,v_2)=:\alpha\in (0,\pi/2]$. 
Assume that $0<\alpha<\pi/2$. 
Then we may construct a second rotational symmetry axis for $\Gamma$
with rotational angle $\varphi_1=\frac{2\pi}{a_1}$,
namely
\[
 v_1^1=\opRot \li \pi,v_2\ri v_1.
\]
Notice that
$$
\ang(v_1,v_1^1)=\min\{2\alpha,\pi-2\alpha\}\in (0,\pi/2],
$$ 
so that in particular
$v_1^1\not= v_1$. So, there are two distinct axes of rotational
symmetry for $\Gamma$ with rotational angle $2\pi/a_1$ with $a_1\ge
3$, contradicting Lemma \ref{lem:number-symm-axes} again.
Therefore, we have $v_1\perp v_2$, which is (1)(i).

Since the first part of the theorem already implies that $\Gamma\cap
v_1=\emptyset$ because $\varphi_1\in (0,\pi)$ it suffices to
show
 $v_2\cap\Gamma\neq\emptyset$ to finally
 establish Part (1)(iii).

Assume that  $v_2\cap\Gamma=\emptyset$, then both axes
$v_1$ and $v_2$ are disjoint from $\Gamma$. Then the rotational 
symmetry is a periodicity, see\cite[p. 256]{burde-zieschang_1985}. 
We denote by $L:=\sL(\Gamma)$ the
length of $\Gamma$. 
The plane $H:=v_1^\perp$ contains $v_2$ according to Part
(1)(i), and we immediately deduce 
that
$H\cap \Gamma\neq\emptyset$ because of the periodicity about $v_2$. 
Fix a point  $x_0\in H\cap \Gamma$, and look at its orbit
$$
O_{v_1}:=\{x_0,...,x_{a_1-1}\}\subset H\cap\Gamma
$$ 
under the action of the rotation $\opRot \li \varphi_1,v_1\ri $ but now -- in contrast to the proof of Lemma \ref{lem:number-symm-axes} -- labelled according to the corresponding arclength parameters.
That is, $x_i=\Gamma(s_i)$ for $i=0,...,a_1-1$ such that
$0\leq s_0 <s_1<...<s_{a_1-1}<L$, and there exists $k\in\N$ with $\gcd(k,a_1)=1$ and unique modulo $a_1$, such that
\begin{equation}\label{eq:rot_xi}
 x_i = \opRot \li \tfrac{2\pi\cdot ki}{a_1},v_1\ri x_0, \quad i=0,...,a_1-1.
 \end{equation}

To justify this, observe first that  periodicity of $\Gamma$ implies that the subarcs on $\Gamma$ connecting
consecutive $x_i$ have equal length, i.e., $s_{i+1}-s_i=L/a_1$ for all $i=0,...,a_1-1$,
and in general 
\begin{equation}\label{eq:length_arcs}
s_j-s_i=\tfrac{L}{a_1}(j-i) \quad 0\leq i\leq j\leq a_1-1.
\end{equation}
Reordering the points in the orbit $O_{v_1}$ according to the rotation counterclockwise, starting
at $y_0:=x_0$ leads to $\{y_0,\ldots,y_{a_1-1}\}$ defined as
$y_j:=\opRot(2\pi j/a_1,v_1)y_0.$ There is an integer $m\in\{1,\ldots,a_1-1\}$ such that $y_1=\Gamma(s_m)=x_m$, so the oriented
subarc
on $\Gamma$ starting at $x_0=y_0$ with endpoint $y_1=x_m$ has
length 
$
s_m-s_0{=}mL/a_1
$
by means of \eqref{eq:length_arcs}. The same holds true for every oriented
subarc from $y_j$ to $y_{j+1}$ for $j=1,\ldots,a_1-1$, so that we arrive at the general relation
\begin{equation}\label{im}
x_{[j\cdot m]}=\Gamma(s_{[j\cdot m]})=y_j=\opRot\li \tfrac{2\pi j}{a_1},v_1\ri y_0
=\opRot\li \tfrac{2\pi j}{a_1},v_1\ri x_0, \quad j=1,\ldots,a_1-1,
\end{equation}
where we denoted $[j\cdot m]= j\cdot m\mod a_1$. If we had $\gcd(m,a_1)>1$
then the  least common multiple $\lcm(m,a_1)$ of $m$ and
$a_1$ could be written as
$\lcm(m,a_1)=m\cdot a_1/\gcd(m,a_1)=:m\cdot n$, where $1<n<a_1-1$ is
a positive integer . Thus, $n\cdot m = 0 \mod a_1$, so that \eqref{im}
 implies $x_{[n\cdot m]}=\Gamma(s_0)=y_n$. But this would mean that 
the remaining points $y_{n+1},\ldots,y_{a_1-1}$ would not be in the orbit
$O_{v_1}$ under the rotation, which is a contradiction.

Hence $\gcd(m,a_1)=1$ so that $m$ possesses an inverse modulo $a_1$, 
i.e., there is a unique $k\in\{1,\ldots,a_1-1\}$ such that $k\cdot m=1 \mod a_1$. Inserting this into \eqref{im} we obtain $x_{[j\cdot m]}=
\opRot\li \tfrac{2\pi j\cdot m\cdot k}{a_1},v_1\ri x_0$ for $j=1,\ldots,a_1-1.$ Given any $i\in\{1,\ldots,a_1-1\}$ we  choose  $j:=i\cdot k$ to  
finally obtain \eqref{eq:rot_xi}.

As $\Gamma$ is $2$-periodic around $v_2\subset H$, there exist 
$\overline{x}_i=\Gamma(\bar{s}_i)=\in\Gamma\cap H$ 
such that 
\begin{equation*}
 \overline{x}_i = \opRot \li \pi,v_2\ri x_i, \quad i=0,...,a_1-1.
\end{equation*}
In terms of arclength on  $\Gamma$ we find $|s_i-\bar{s}_i|=L/2$ 
for
each $i=0,\ldots,a_1-1$.

By a short calculation, e.g., by means of the matrix representations
of $\opRot \li \pi,v_2\ri $ and $\opRot \li 2\pi ki/a_1,v_1\ri $ with respect to
an orthonormal basis containing the unit vectors through $v_1$ and $v_2$, we arrive at 
\begin{equation}\label{eq:rot_barxi}
 \overline{x}_i = \opRot \li \tfrac{2\pi\cdot k(-i)}{a_1},v_1\ri \overline{x}_0, \quad i=0,...,a_1-1.
\end{equation}
Next, we consider the circle 
$S:=\partial B_r(0)\cap H$ with $r:=\dist(x_0,0)$. We have 
$x_i,\overline x_i\in S$ for all $i=0,...,a_1-1$. 
We are going  to determine the order of these points on $S$, 
and consider first 
only the $x_i$. 
Due to the $a_1$-periodicity, there is a unique successor $x_{i_k}$ of $x_0$ 
(counterclockwise) on $S$ which has a distance of $2\pi r/a_1$ to $x_0$ 
on $S$ and is defined by \eqref{eq:rot_xi}:
\[
 x_{i_k}=\opRot \li \tfrac{2\pi\cdot ki_k}{a_1},v_1\ri x_0 = \opRot \li \tfrac{2\pi\cdot 1}{a_1},v_1\ri x_0
\]
which is equivalent to $ki_k\equiv_{a_1} 1$. Thus $i_k$ is the unique inverse of $k$ in $\Z/a_1\Z$ 
which exists as $\gcd(k,a_1)=1$. Repeating this argument for the other successors, 
we arrive at the order 
\begin{equation}\label{chain}
 x_0 - x_{i_k} - x_{2i_k} - \cdots - x_{(a_1-1)i_k}.
\end{equation}
In an analogous way we arrive by using \eqref{eq:rot_barxi} at the 
following (counterclockwise) order for the $\overline x_i$, 
$i=0,...,a_1-1$ on the circle $S$:
\begin{equation}\label{barchain}
 \overline x_0 - \overline x_{(a_1-1)i_k} - \overline x_{(a_1-2)i_k} 
 - \cdots - \overline x_{i_k}.
\end{equation}
On $S$ we have 
\begin{equation}\label{eq:lengthstuff}
 \sL \li a_S \li x_i,x_{i+li_k}\ri \ri = 2\pi rl/a_1 = \sL \li a_S \li \overline x_i,\overline x_{i-li_k}\ri \ri ,
\end{equation}
where $a_S(x,y)$ is the circular subarc
of $S$ connecting $x$ and $y$ counterclockwise.
Now we are going to  determine 
the order on $S$ of both sets of points combined. To this extent, 
we consider a pair 
$(x_j,\overline x_j)$ such that $x_j$ 
minimizes
$\dist \li x_k,v_2\cap S\ri $ for $k=0,...,a_1-1$. 
Without loss of generality let this be $j=0$ and assume further without loss of generality 
that $a_S \li x_0,\overline x_0\ri \leq a_S \li \overline x_0,x_0\ri $.
Now we claim 
\begin{equation}\label{betaclaim}
\beta := \sL \li a_S \li x_0,\overline{x}_0\ri \ri <
2\pi r/a_1.
\end{equation}
Indeed, if $\beta>2\pi r/a_1$, then \eqref{eq:lengthstuff} implies
 $x_{i_k}\in a_S \li x_0,\overline x_0\ri $ 
and therefore $\dist \li x_{i_k},v_2\cap S\ri < \dist \li x_0,v_2\cap S\ri $,
which contradicts the minimality of $x_0$.
If $\beta=2\pi r/a_1$, then $x_{i_k}=\overline x_0$, 
and for the lengths of the connecting subarcs on $\Gamma$
we have 
\[
 L/2=|s_0-\bar{s}_0| = |s_0-s_{i_k}|=s_{i_k}-s_0
 \overset{\eqref{eq:length_arcs}}{=} \frac{L}{a_1}i_k.
\]
If $a_1$ is odd, this is a contradiction straight away. 
If $a_1$ is even, then $i_k=\frac{a_1}{2}>1$ since $a_1\ge 3$, and 
thus $\gcd(i_k,a_1)= i_k>1$.
But recall that 
$i_k$ satisfies $ki_k\equiv_{a_1} 1$, i.e.,  
$k$ is the unique inverse to $i_k$ in $\Z/a_1\Z$, 
which exists if and only if 
$\gcd(i_k,a_1)=1$, contradiction. 
Therefore, our claim \eqref{betaclaim} is proven.

Combining \eqref{betaclaim} with \eqref{eq:lengthstuff}
leads to
the counterclockwise ordered combined chain 
\begin{equation}\label{combinedchain}
 x_0 - \overline x_{0} - x_{i_k} - \overline x_{(a_1-1)i_k} - 
 x_{2i_k} - \overline x_{(a_1-2)i_k} - \cdots - x_{(a_1-1)i_k} - 
 \overline x_{i_k},
\end{equation}
since there are no $x_i,\overline x_i$ in the circular arc
$a_S \li x_0,\overline x_0\ri \subset S $ because of the minimality
of $x_0$, and the possible successors of $x_0$ and $\overline{x}_0$, respectively, 
are $x_{i_k}$ and $\overline x_{(a_1-1)i_k}$. 
Equation \eqref{eq:lengthstuff} delivers that $x_{i_k}$ has to appear before $\overline x_{(a_1-1)i_k}$. From there one can continue to 
form the whole combined chain \eqref{combinedchain}.

The $a_1$-periodicity now gives us   information
on the shorter subarcs  $a(p,q)\subset\Gamma$ connecting 
consecutive points $p$ and $q$ on the
combined chain \eqref{combinedchain}:
\[
a \li x_{li_k},\overline x_{(a_1-l)i_k}\ri =\opRot \li 2\pi l/a_1, v_1\ri a \li x_0,\overline x_0\ri \Foa l\in\N.
\]
In particular, 
the lengths of these arcs coincide. But this leads to 
\[
 \sL \li a \li x_{i_k},\overline x_{i_k}\ri \ri  = |s_{k_k}-\bar{s}_{i_k}| =
 L/2 = |s_0-\bar{s}_0| =\sL \li a \li x_0,\overline x_0\ri \ri  = \sL \li a \li x_{i_k},
 \overline x_{(a_1-1)i_k}\ri \ri,
\]
and therefore $1=a_1-1$, which is not the case as $a_1\geq 3$.
This final contradiction leads us to  
$v_2\cap\Gamma\not=\emptyset$. 
This establishes (1)(iii) and concludes the whole proof.\qed
%


\begin{appendix}
\section{Estimates for arclength parametrizations}\label{sec:arclength}

At the beginning of the proof of the second part of Theorem \ref{thm:blatt} we have 
used the following lemma stating that the (finite) Sobolev-Slobodetckij norm is conserved (up to 
constants) if one reparametrizes a regular absolutely continuous curve to arclength.
Note that we have assumed $\alpha >2$ in that part of Theorem \ref{thm:blatt}, so that we state
this auxiliary lemma in the range of Sobolev exponents that allow for a continuous embedding into
classic function spaces with H\"older continuous first 
derivatives; cf. Remark \ref{rem:sobolev-slobo}.

\begin{lem}\label{lem:arclength-seminorm}
Assume that $\g\in W^{1+s,\rho}(\R/\Z,\R^n)$ for $\rho\in (1,\infty)$ and $s\in (1/\rho,1)$, and that
 $|\g'|>0 $ on $\R/\Z$. Then the corresponding arclength parametrization
 $\Gamma$ is of class $W^{1+s,\rho}(\R/(L\Z),\R^n)$ satisfying the estimate
 \begin{equation}\label{seminormest-arclength}
 [\Gamma]^\rho_{W^{1+s,\rho}}\le \Big( \frac 1c \Big)^{1+(s+1)\rho}\Big[
 1+\Big(\frac Cc \Big)^\rho \Big]C^2\cdot
 [\gamma]^\rho_{W^{1+s,\rho}}, 
 \end{equation}
 where $L:=\sL(\g)$ denotes the positive and finite length of $\g$, and $c:=\min_{[0,1]}|\g'|,$
 $C:=\max_{[0,1]}|\g'|.$ 
 \end{lem}
 \begin{proof}
 Since $W^{1+s,\rho}(\R/\Z,\R^n) $ continuously embeds into $C^{1,s-(1/\rho)}(\R/\Z,\R^n)$
 we have
 \begin{equation}\label{tangentbounds}
 c:=\min_{[0,1]}|\g'|\le |\g'(\tau)|\le \max_{[0,1]}|\g'|=:C \Foa\tau\in [0,1],
 \end{equation}
 so that the arclength parameter $s(t):=\int_0^t|\g'(\tau)|\,d\tau$  is a bi-Lipschitz continuous
 function $s:[0,1]\to [0,L]$ with 
 \begin{equation}\label{bilip-s}
 c|t_1-t_2|\le |s(t_1)-s(t_2)|\le C|t_1-t_2| \Foa t_1,t_2\in [0,1],
 \end{equation}
 and its inverse function $t:=s^{-1}:[0,L]\to [0,1]$  satisfies
 \begin{equation}\label{bilip-t}
 \frac 1C |s_1-s_2|\le |t(s_1)-t(s_2)|\le \frac 1c |s_1-s_2| \Foa s_1,s_2\in [0,L],
 \end{equation}
 Moreover, using \eqref{tangentbounds} for the derivative $t'(s)=1/|\g'(t(s))|$ one has 
 \begin{equation}\label{inversetangentbounds}
 \frac 1C \le |t'(s)|\le \frac 1c \Foa s\in [0,L].
 \end{equation}
Now we start estimating the seminorm of the arclength parametrization $\Gamma(\cdot)=
\g\circ t(\cdot)$.
\begin{align}
[\Gamma]^\rho_{W^{1+s,\rho}} & =  \int_{\R/(L\Z)}\int_{-L/2}^{L/2}
\frac{|\Gamma'(u+w)-\Gamma'(u)|^\rho}{|w|^{1+s\rho}}\,dwdu\notag\\
& = 
\int_{\R/(L\Z)}\int_{-L/2}^{L/2}
\frac{|\gamma'(t(u+w))t'(u+w)-\gamma'(t(u))t'(u)|^\rho}{|t(u+w)-t(u)|^{1+s\rho}}\cdot
\frac{|t(u+w)-t(u)|^{1+s\rho}}{|w|^{1+s\rho}}\,dwdu\notag\\
& \le  \int_{\R/(L\Z)}\int_{-L/2}^{L/2}
\frac{|\gamma'(t(u+w))-\gamma'(t(u))|^\rho |t'(u+w)|^\rho}{|t(u+w)-t(u)|^{1+s\rho}}
\cdot\frac{|t(u+w)-t(u)|^{1+s\rho}}{|w|^{1+s\rho}}\,dwdu\notag\\
&  + 
\int_{\R/(L\Z)}\int_{-L/2}^{L/2}
\frac{|\gamma'(t(u))|^\rho |t'(u+w)-t'(u)|^\rho}{|t(u+w)-t(u)|^{1+s\rho}}
\cdot\frac{|t(u+w)-t(u)|^{1+s\rho}}{|w|^{1+s\rho}}\,dwdu.\label{prelimest}
\end{align}
By means of \eqref{inversetangentbounds} and \eqref{bilip-t} we can estimate the first double 
integral on the right-hand side of \eqref{prelimest} by
\begin{equation}\label{trafo-seminorm1}
\Big(\frac 1c \Big)^{1+(s+1)\rho}\int_{\R/(L\Z)}\int_{-L/2}^{L/2}
\frac{|\gamma'(t(u+w))-\gamma'(t(u))|^\rho}{|t(u+w)-t(u)|^{1+s\rho}}
\,dwdu.
\end{equation}
With
the help of \eqref{inversetangentbounds}  we find
$$
|t'(u+w)-t'(u)|=\Big| \frac{1}{|\g'(t(u+w))|}-
\frac{1}{|\g'(t(u))|}\Big|\overset{\eqref{inversetangentbounds}}{\le}
c^{-2}|\g'(t(u))-\g'(t(u+w))|, 
$$
and this combined with \eqref{bilip-t} gives for  the second double integral on the right-hand
side of \eqref{prelimest} the upper bound
\begin{equation}\label{trafo-seminorm2}
\left(\frac{C}{c^2}\right)^{\rho} \Big(\frac 1c \Big)^{1+s\rho}\int_{\R/(L\Z)}\int_{-L/2}^{L/2}
\frac{|\gamma'(t(u+w))-\gamma'(t(u))|^\rho}{|t(u+w)-t(u)|^{1+s\rho}} \,dwdu.
\end{equation}
The integrals in \eqref{trafo-seminorm1} and \eqref{trafo-seminorm2} are identical and may
be transformed using first the substitution $z(w):=t(u+w)$ with 
$$
dz(w)=t'(u+w)dw=\frac{1}{|\g'(t(u+w))|}dw=\frac{1}{|\g'(z)|}dw
$$
for the $w$-integration, giving
$$
\int_{\R/(L\Z)}\int_{t(u-(L/2))}^{t(u+(L/2))}
|\g'(z)|\frac{|\gamma'(z)-\gamma'(t(u))|^\rho}{|z-t(u)|^{1+s\rho}}
\,dzdu.
$$
Due to the $1$-periodicity the inner integral can be replaced by the integration
over $\R/\Z$, and after applying Fubini's theorem we may change variables
according to  $y(u):=t(u)$ for the integration with respect to $u$ with $dy(u)=|\g'(y)|^{-1}du$, 
which by virtue of \eqref{tangentbounds} leads to 
\begin{equation}\label{expression}
\int_{\R/\Z}\int_{\R/\Z}
|\g'(y)||\g'(z)|\frac{|\gamma'(z)-\gamma'(y)|^\rho}{|z-y|^{1+s\rho}}
\overset{\eqref{tangentbounds}}{\le}C^2[\g]^\rho_{W^{1+s,\rho}}.
\end{equation}
Recall that \eqref{expression} 
serves as an upper bound for the double integral that appears both in 
 \eqref{trafo-seminorm1}, and in \eqref{trafo-seminorm2}. So, combining this with \eqref{prelimest}  
 leads to the desired estimate
$$
[\Gamma]^\rho_{W^{1+s,\rho}}\le \Big( \frac 1c \Big)^{1+(s+1)\rho}\Big[
 1+\Big(\frac Cc \Big)^\rho \Big]C^2\cdot
 [\gamma]^\rho_{W^{1+s,\rho}}.
 $$
 \end{proof}

With 
a simple argument (similar to the one in 
\cite[Lemma 4.2]{strzelecki-etal_2009}) we now show
that injective  curves  parametrized by arclength of class $C^{1,\mu}$
are bi-Lipschitz.
\begin{lem}\label{lem:arclength-bilip}
Let $\mu\in (0,1]$, $L>0$,  and $\Gamma\in C^{1,\mu}(\R/(L\Z),\R^n)$ with $|\Gamma'|\equiv 1$ on
$[0,L]$, such that $\Gamma|_{[0,L)}$ is injective.
 Then there is a constant $B=B(\mu,\Gamma)\ge 1$ such that 
 $$
 \frac{1}{B}|w|\le |\Gamma(u+w)-\Gamma(u)|\le |w|\Foa u\in \R/(L\Z),\, |w|\le L/2.
 $$
 \end{lem}
 From the Morrey-type embedding mentioned in Remark \ref{rem:sobolev-slobo} and the 
 specification in \eqref{embedding} we directly derive the following corollary.
 \begin{cor}\label{cor:arclength-bilip}
 Let $L>0$, $\rho\in (1,\infty)$, $s\in (1/\rho,1)$, and $\Gamma\in W^{1+s,\rho}(\R/(L\Z),\R^n)$
 be an injective arclength parametrized curve. Then there is a constant $B=B(s,\rho,\Gamma)\ge 1$
 such that
\begin{equation}\label{arclength-bilip1}
 \frac{1}{B}|w|\le |\Gamma(u+w)-\Gamma(u)|\le |w|\Foa u\in \R/(L\Z),\, |w|\le L/2.
\end{equation} 
 In particular, there is a constant $B=B(\alpha,\Gamma)$ such that any injective arclength 
 parametrized curve  $\Gamma\in W^{(\alpha +1)/2,2}(\R/(L\Z),\R^n)$ satisfies 
 \eqref{arclength-bilip1}.
 \end{cor}
{\it Proof of Lemma \ref{lem:arclength-bilip}.}\,
We only need to prove the left inequality of the bi-Lipschitz estimate since the upper bound
follows from $|\Gamma'|\equiv 1$ on $[0,L]$.
Without loss of generality we may assume that $\Gamma'(u)=(1,0\ldots,0)\in\R^n$ so that we may estimate
the tangent's first component $\Gamma_1'$ from below as
\begin{align*}
\Gamma_1' (u+w) & \ge \Gamma_1'(u)-|\Gamma_1'(u)-\Gamma_1'(u+w)|\\
& \ge 1-\|\Gamma\|_{C^{1,\mu}}|w|^\mu\ge \frac 34 \quad\Foa |w|\le 
\eps_0:=\Big( \frac{1}{4\|\Gamma\|_{C^{1,\mu}}}\Big)^{1/\mu},
\end{align*}
which implies
\begin{equation}\label{prelim-bilip1}
|\Gamma(u+w)-\Gamma(u)|\ge |\Gamma_1(u+w)-\Gamma_1(u)|=\Big|\int_u^{u+w}\Gamma_1'(\tau)\,d\tau\Big|
\ge \frac 34 |w| \Foa |w|\le\eps_0.
\end{equation}
The continuous function $g(u,w):=|\Gamma(u+w)-\Gamma(u)|$, on the other hand, 
is uniformly continuous on the compact set 
$$
\Sigma:=\{(u,w)\in \R/(L\Z)\times [-L/2,L/2]: |w|\ge\eps_0\},
$$
and $g$ is strictly positive on $\Sigma$ 
since $\Gamma|_{[0,L)}$ is assumed to be injective. Hence there is a positive
constant $c=c(\Gamma)$ such that $g|_\Sigma\ge c$, which implies
\begin{equation}\label{prelim-bilip2}
|\Gamma(u+w)-\Gamma(u)|\ge c\ge \frac{2 c}L |w|
 \Foa \eps_0\le |w|\le L/2.
\end{equation}
Combining \eqref{prelim-bilip1} with \eqref{prelim-bilip2} we obtain the desired
bi-Lipschitz estimate for the constant $B=B(\mu,\Gamma):=\max\{\frac 43 , \frac{L}{2c}\}$.
\qed

In the proof of  part (ii) of Theorem \ref{thm:blatt} we have also used the following elementary
inequality.
\begin{lem}\label{lem:elem-ineq}
For any $\alpha\in (1,\infty)$ one has
$$
1-x^\alpha\le (\alpha+1)(1-x)\Foa x\in [0,1].
$$
In particular, if $\alpha\in [2,\infty)$, the following holds.
$$
1-x^\alpha\le (\alpha+1)(1-x^2)\Foa x\in [0,1].
$$
\end{lem}
\begin{proof}
It suffices to prove that the function $f_\alpha(x):=x^\alpha-(\alpha+1)x+\alpha$ is non-negative
for all $x\in [0,1]$, and for all $\alpha\in (1,\infty), $ since $f$ may be rewritten as
$$
f(x)=x^\alpha+(\alpha+1)(1-x)-1.
$$
One immediately checks for the derivative (which exists as $\alpha >1$)
$$
f'(x)=\alpha x^{\alpha -1} - (\alpha+1)\le -1\Foa x\in [0,1],
$$
so that $f$ strictly decreases from the positive value $f(0)=\alpha$
to the value $f(1)=0$ on $[0,1]$.
\end{proof}

Lemma \ref{lem:symmetric-subset}   requires the existence of some $k\in\Z$
satisfying specific equivalence class relations, established in the following
elementary result.
\begin{lem}\label{lem:restklasse}
For relatively prime numbers $a,b\in\Z\setminus\{0,\pm 1\}$ and some $m\in\N$, $m >1$, dividing
either $a$ or $b$, there is an integer  $k\in\Z$, which is unique modulo $m$,  such that
\begin{equation}\label{kk}
\begin{cases}
[ak+1]=e=[0]\in \Z/m\Z & \textnormal{if $m|b$}\\
[bk+1]=e=[0]\in \Z/m\Z & \textnormal{if $m|a$}.
\end{cases}
\end{equation}
\end{lem}
\begin{proof}
It suffices to treat the case $m|b$. The required condition 
$[ak+1]=[0]$ (identifying $k$ uniquely modulo $m$)
is equivalent
to $[ak]=[-1]$ or $[(-a)k]=[1]$, which means that $(-a)$ is invertible modulo $m$, or, 
equivalently that $(-a)$ and $m$ are relatively prime.
Assuming that there is a
common divisor $d\in\Z$, $|d|\ge 2$ of $(-a)$ and $m$, then $d$ divides also $b$ since 
$m|b$, but this contradicts our assumption that $a$ and $b$ are relatively prime.
\end{proof}

For the proof of Theorem \ref{thm:twocritical} we needed the following 
elementary number theoretical result.
\begin{lem}\label{lem:gcd}
If  two integers $a,b\in\Z\setminus\{0\}$  are relatively prime then also
the two integers $a+b$ and $ab$.
\end{lem}
\begin{proof}
Let $\gcd(a,b)=1$. Assuming that $a+b$ and $ab$ are \emph{not} relatively prime, we can find
a prime $n$ such that $n|(a+b)$ and $n|ab$. 
The second condition implies $n|a$ or $n|b$. 
Without loss of generality we assume $n|a$, as for $n|b$ the argumentation is analogous. 
Combining $n|a$ with $n|(a+b)$, we arrive at $n|b$, which 
contradicts  $\gcd(a,b)=1$. 
\end{proof}

In the proof of Theorem \ref{thm:gruenbaum-gilsbach} we used the 
following simple result concerning images of rotationally symmetric sets
under isometries of $\R^3$.
\begin{lem}\label{lem:rotationsets}
Let $v\in \S^2$, $\beta\in\R$, and $I:\R^3\to\R^3$ an orientation
preserving isometry of $\R^3$ with  $I(v)\not= 0$.
Then  for any set
$M\subset\R^3$ with
\begin{equation}\label{rotvor}
\opRot \li \beta,\R v\ri M=M
\end{equation}
one has
\begin{equation}\label{Rrot}
\opRot \li \beta,
 I(\R v)\ri I(M)=I(M),
\end{equation}
where similarly as before $\opRot \li \beta,w\ri $ stands for the 
rotation about the 
affine line  $w=\R e_w+d\subset\R^3$ for some   $e_w \in \S^2$
and $d\in\R^3$
with rotational angle $\beta\in \R.$
(For $\beta >0$ with $\beta\not\in\pi\Z$ and any $\xi\not\in w$, the set 
$$
\mathscr{B}:= 
\{\xi-\Pi_w \li \xi\ri ,\opRot \li \beta,w\ri \xi-\Pi_w \li \opRot \li \beta,w\ri \xi\ri ,e_w\}
$$ 
forms a 
positively oriented\footnote{That is, the $3\times 3$-matrix mapping $\mathscr{B}$ onto the standard basis
$\{e_1,e_2,e_3\}$ has positive determinant.} basis of
$\R^3$. Here $\Pi_w$ denotes the orthogonal projection
onto the affine line $w$.)
\end{lem}
\begin{proof}
The statement is trivial for $\beta=0$ since the rotations involved are simply
the identity mapping in $\R^3$.
Without loss of generality let $\beta>0$. 
For $Y\in I(M)$ there is exactly one $\eta\in M$ such that
$Y=I(\eta).$ 
Exploiting assumption \eqref{rotvor} we find exactly one
 $\xi\in M$ such that $\eta=\opRot \li \beta,\R v\ri \xi$. Let
$\xi_0:= \Pi_{\R v}(\eta) $
be the orthogonal projection of $\eta$ onto the rotational axis 
$\R v$ so that 
$$
\xi_0\in\R v,\quad (\xi-\xi_0)\perp v,\quad 
(\eta-\xi_0)\perp v,\quad  |\xi-\xi_0|=|\eta-\xi_0|
$$
and such that the set $\mathscr{C}:=\{\xi-\xi_0,\eta-\xi_0,v\}$ forms
a positively oriented basis of $\R^3$ if $\eta\not\in\R v$ and $\beta\not\in\pi\Z$.
Since $I$ is an orientation preserving isometry we can write
$Ix=Sx+b$, $x\in\R^3$, for some $S\in SO(3)$ and $b\in\R^3$, and
find 
$
I(\xi_0)\in I(\R v)$ and 
$$
(I(\xi)-I(\xi_0))\perp Sv,\quad 
(I(\eta)-I(\xi_0))\perp Sv,\quad  |I(\xi)-I(\xi_0)|=|I(\eta)-I(\xi_0)|,
$$
and the set $\mathscr{D}:=\{I(\xi)-I(\xi_0),I(\eta)-I(\xi_0),Sv\}$ 
forms a
positively oriented basis of $\R^3$.  In addition, by isometry,
$$
\cos\beta=\frac{(\xi-\xi_0)\cdot(\eta-\xi_0)}{|\xi-\xi_0||\eta-\xi_0|}=
\frac{(I(\xi)-I(\xi_0))\cdot(I(\eta)-I(\xi_0))}{|I(\xi)-I(\xi_0)||I(\eta)-I(\xi_0)|},
$$
which is also true for $\beta\in\pi\Z$, so that for $X:=I(\xi)$ we arrive at 
$$
\opRot \li \beta,I(\R v)\ri X=
\opRot \li \beta,I(\R v)\ri I(\xi)=I(\eta)= Y,
$$
which proves the inclusion 
\begin{equation}\label{inclusion}
I(M)\subset \opRot \li \beta,I(\R v)\ri 
I(M)\quad\textnormal{for arbitrary 
$\beta\in
\R\setminus\{0\},$} 
\end{equation}
since the same argument works for $\beta <0$ only with negatively oriented
bases $\mathscr{C}$ and $\mathscr{D}.$
The inclusion \eqref{inclusion} is trivial  if
$Y=I(\eta)$ for some $\eta\in\R v$ because  then
 $\opRot \li \beta,I(\R v)\ri I(\eta)=I(\eta)$. 
Since we proved \eqref{inclusion} for arbitrary $\beta\not= 0$ 
we can apply the inverse rotation
$\opRot \li \beta,I(\R v)\ri ^{-1} = \opRot \li -\beta,I(\R v)\ri $ 
to \eqref{inclusion}
and use the above argument again.
\end{proof}

\begin{lem}\label{lem:commut_rotation}
  Let $A\in SO(3)$ 
  be a rotational matrix with angle $\phi=2\pi/b$, $b\in\N$, about 
  the $z$-axis and $M\subset\R^3$ 
  be a set invariant with respect to said rotation, i.e. $AM=M$. 
  For any rotational matrix $B\in SO(3)$ about an axis $v$ with $v\perp e_3$, 
  $v\cap \R e_3=\{0\}$, and rotational
  angle $\pi$ we have 
  \[
  ABM = BAM =BM.
  \]
\end{lem}
\begin{proof}
  The case $b=1$ is trivial. Therefore let $\phi=2\pi/b$, $b\geq 2$, be the rotational angle of $A$, 
  and $e_v\in\S^2$ be a unit vector
  contained in $v$, and set $f:=e_3\wedge e_v$. The matrix
  representations of $A$ and $B$ with respect to  the orthonormal
  basis 
  $
  \mathscr{B}:=\{e_v,f,e_3\}$ are given by
  \[
   A = \begin{pmatrix}
        \cos \phi &-\sin \phi &0\\
        \sin\phi & \cos\phi& 0 \\
        0 & 0 & 1
       \end{pmatrix}, \quad 
   B = \begin{pmatrix}
        1 &  0 &  0 \\
        0 & -1 &  0 \\
        0 &  0 & -1 \\
       \end{pmatrix}.
  \]
  Further, the assumption $AM=M$ implies 
  \[
    y:= A^kx\in M \Foa x\in M,\,k\in\Z/(b\Z).
  \]
  Hence it suffices to show that there is $k\in\Z/(b\Z)$ such that
  \begin{equation}\label{eq:rotprop}
   ABx = BA A^k x\Foa x\in M
  \end{equation}
  
to prove the inclusion
  $ABM\subset BAM$. On the other hand, if \eqref{eq:rotprop} is 
  established  
 for some $k\in\Z/(b\Z)$ then we can use our assumption
  $AM=M$, hence also $A^kM=M$\, 
again to write any $y\in M$ as $A^{k} x=y$ for an 
  appropriate $x\in M$, so that \eqref{eq:rotprop} implies
  also the reverse inclusion $BAM\subset ABM.$

  To establish \eqref{eq:rotprop} we
   calculate for $x=(x^1,x^2,x^3)\in M$ 
  \begin{align*}
   ABx & = \begin{pmatrix}
   \cos\phi & -\sin\phi & 0\\
   \sin\phi & \cos\phi & 0\\
   0 & 0 & 1\\
   \end{pmatrix}
	    \begin{pmatrix}
	      1 &  0 &  0 \\
	      0 & -1 &  0 \\
	      0 &  0 & -1 \\
	    \end{pmatrix}x 
	 = \begin{pmatrix}
	    \cos\phi & -\sin\phi & 0\\
	       \sin\phi & \cos\phi & 0\\
	          0 & 0 & 1\\
		     \end{pmatrix}
		     \begin{pmatrix} x^1 \\ -x^2 \\ -x^3 \end{pmatrix} \\
       & = \begin{pmatrix} x^1\cos  \li 2\pi/b\ri +x^2\sin  \li 2\pi/b\ri  \\ x^1\sin \li 2\pi/b\ri -x^2\cos \li 2\pi/b\ri  \\ -x^3 \end{pmatrix}
  \end{align*}
  as well as
  \begin{align*}
   BA A^k x = BA^{k+1}x  & =  \begin{pmatrix}
	      1 &  0 &  0 \\
	      0 & -1 &  0 \\
	      0 &  0 & -1 \\
	    \end{pmatrix}
	    \begin{pmatrix}
	    \cos\phi & -\sin\phi & 0\\
	    \sin\phi & \cos\phi & 0\\
	    0 & 0 & 1\\
	    \end{pmatrix}^{k+1}\begin{pmatrix}x^1\\ x^2\\ x^3\\
	    \end{pmatrix}
	     \\
	 & =  \begin{pmatrix}
	      1 &  0 &  0 \\
	      0 & -1 &  0 \\
	      0 &  0 & -1 \\
	    \end{pmatrix} 
	    \begin{pmatrix}
                \cos \li (k+1)\phi\ri  & -\sin \li (k+1)\phi\ri  & 0\\
	            \sin \li (k+1)\phi\ri  & \cos \li (k+1)\phi\ri  & 0\\
	                0 & 0 & 1\\
						            \end{pmatrix}\begin{pmatrix}x^1\\x^2\\x^3\\ \end{pmatrix}
	     \\
       & =  \begin{pmatrix} x^1\cos  \li 2\pi(k+1)/b\ri -x^2\sin  \li 2\pi(k+1)/b\ri  \\ -x^1\sin \li 2\pi(k+1)/b\ri -x^2\cos \li 2\pi(k+1)/b\ri  \\ -x^3 \end{pmatrix} 
  \end{align*}
  Due to the symmetry properties of sine and cosine we arrive at 
 \eqref{eq:rotprop} if and only if $k+1\equiv_b -1$ or $k=-2 \mod b$.
 
\end{proof}

\end{appendix}
  \section*{Acknowledgments}

Part of this work is contained in the first author's Ph.D. thesis \cite{gilsbach_2018}. 
       The second author's work is partially funded by DFG Grant no.
       Mo 966/7-1 \emph{Geometric curvature functionals: energy landscape and discrete
       methods} and by the
         Excellence Initiative of the German federal and state 
	 governments. 

  \bibliography{../bib-files/refs-bookproj}{}
\bibliographystyle{acm}

		 \end{document}